\documentclass[10pt]{article}
\usepackage{subfigure}
\usepackage{booktabs,multirow}
\usepackage{amssymb,amsmath,amsthm}
\usepackage[numbers]{natbib}
\usepackage{tikz}
\usepackage{mathrsfs}

\usepackage[all]{xy}
\usepackage{float}
\usepackage{graphicx}
\usepackage[overload]{empheq}
\usepackage{algorithmic}

\usepackage{framed} 
\usepackage{lastpage}
\usepackage[algoruled]{algorithm2e}
\usepackage{color,soul}
\usepackage{authblk}

\usepackage[pagebackref=false,colorlinks,linkcolor=blue,citecolor=magenta]{hyperref}
\usepackage[nottoc]{tocbibind}
\newcommand{\bm}{\mathbf}
\let\oldnl\nl% Store \nl in \oldnl
\newcommand{\nonl}{\renewcommand{\nl}{\let\nl\oldnl}}

\newtheorem{theorem}{Theorem}[section]
\newtheorem{definition}{Definition}[section]
\newtheorem{corollary}{Corollary}[theorem]
\newtheorem{lemma}[theorem]{Lemma}
\newtheorem{assumption}{Assumption}[section]
\newtheorem{remark}{Remark}[section]

\title{A projected descent method for minimization of weakly semismooth functions over polyhedral sets}

\author[1]{Morteza Maleknia}
\author[2]{Majid Soleimani-damaneh}

\affil[1]{Department of
	Mathematical Sciences, Isfahan University of Technology, Isfahan 84156-83111, Iran}

\affil[2]{School of Mathematics, Statistics and Computer Science, College of Science, University of
	Tehran, Tehran, Iran}

\affil[ ]{\textit {\{m.maleknia@iut.ac.ir,m.soleimani.d@ut.ac.ir\}}}

\begin{document}

\maketitle

	\section*{Abstract}
	This study develops a projection-based descent method for minimizing weakly semismooth functions over polyhedral sets, taking an initial step toward extending projection-based descent algorithms to nonsmooth, nonconvex optimization problems.
	At each iteration, the method considers an inner approximation of the Clarke $\varepsilon$-subdifferential at the current iterate over the feasible region and uses its least-norm element to generate a search direction.
	An exponential limited backtracking line search is then developed to assess the quality of the generated direction, while a projected variant of Mifflin's line search is proposed to identify nonredundant subgradients that enrich the Clarke $\varepsilon$-subdifferential approximation when necessary.
	To quantify stationarity, a characterization of stationary points based on the projection operator yields a computable optimality measure for the proposed method.
	The convergence properties of the proposed method are established under mild assumptions. If the algorithm generates infinitely many serious steps, we identify two subsequences such that every cluster point of each subsequence is stationary. If the number of serious steps is finite, we show that the final serious step generates a stationary point.
	Numerical experiments demonstrate the efficiency and broad applicability of the proposed method across a wide range of test problems, including applications in image denoising, multiobjective optimization, and data clustering. To extend the applicability of the proposed method to problems with general smooth nonlinear constraints, we propose a heuristic approach based on sequential linearization.

	\section{Introduction}
	The problem of minimizing a real-valued locally Lipschitz function over a closed subset of $\mathbb{R}^n$ is a challenging problem with broad applications in various fields. Following the introduction of the Clarke subdifferential \cite{Clarke1990}, considerable effort has been devoted to developing numerical algorithms that possess global convergence guarantees and admit termination criteria based on necessary optimality conditions. In particular, the concept of the Clarke $\varepsilon$-subdifferential \cite{Gold1} has played a significant role in constructing stabilized search directions and deriving computable optimality measures that serve as reliable termination criteria for nonsmooth optimization algorithms.
	
	Let us consider the constrained optimization problem
	\begin{equation}\label{Intro-main1}
		\min \,\, f(\bm x) \quad \text{s.t.} \quad g_i(\bm x)\leq 0, \quad i\in\{1,2,\ldots,m\},
	\end{equation}
	where $f:\mathbb{R}^n\to\mathbb{R}$ and $g_i:\mathbb{R}^n\to\mathbb{R}$, $i\in\{1,2,\ldots,m\}$, are nonsmooth, but locally Lipschitz. For ease of exposition, we denote the feasible region of problem~\eqref{Intro-main1} by 
	$$
	C:=\big\{\bm x\in\mathbb{R}^n:\ g_i(\bm x)\leq 0,\ i\in\{1,2,\ldots,m\}\big\}.
	$$
	A widely used strategy for handling the constraints in problem \eqref{Intro-main1} is the exact penalty method \cite{Bagirov2014}, which incorporates the constraint violations into the objective function through an exact penalty term. This reformulation gives rise to a sequence of unconstrained nonsmooth optimization problems that can then be solved by algorithms such as subgradient methods \cite{Amir_Beck_first_order,bagirov2012,Maleknia-Optimization}, bundle methods \cite{kiwielbook}, and gradient sampling methods \cite{Burke2005,GS-full}.
	Another approach, commonly used in bundle-type methods, is to incorporate the constraint functions into an auxiliary function, often referred to as the improvement function. At each iteration, based on local linearization techniques, the corresponding improvement function is minimized to obtain a feasible descent direction \cite{kiwielbook}. In the context of gradient sampling methods, a sequential quadratic programming framework has been developed to handle constrained problems \cite{Curtis2012}. When the objective function $f$ is quasi-differentiable and the constraints are linear, a derivative-free method based on discrete gradients was proposed in \cite{free-derivative}. After introducing slack variables for the inequality constraints, the method eliminates the constraints by expressing the nonbasic variables in terms of the basic variables.
	
	When the feasible region $C\subset\mathbb{R}^n$ is a closed convex set, the constraints can be handled by means of the orthogonal projection onto $C$. For many simple sets, including the Cartesian product of closed intervals, polyhedral sets, and ellipsoids, the projection step can be computed efficiently, making projection-based methods an attractive approach for constraint handling, especially in large-scale optimization problems \cite{large-scale-Newton}. The projected subgradient method \cite{shorbook} is one of the simplest approaches for minimizing a nonsmooth function over a closed convex set. It does not require solving a subproblem to compute a search direction and relies on a prescribed sequence of step sizes. A notable generalization of the projected subgradient method is the mirror descent method \cite{Amir_Beck_first_order}, which employs the Bregman distance to define the proximity term. Despite their simplicity, these methods are not descent methods and do not inherently provide a termination criterion based on necessary optimality conditions. Furthermore, their convergence guarantees are generally developed for convex optimization problems. Along this line of research, Kiwiel employed the projection technique within the bundle framework and developed a projection-based descent algorithm for nonsmooth convex optimization problems \cite{Kiwiel-projection}. The convergence analysis of that method, however, is restricted to convex objective functions. An attempt to extend projection-based descent methods to the minimization of nonconvex locally Lipschitz functions over bound-constrained problems was made in \cite{Karmitsa-projection}. However, establishing a convergence theory for the proposed method remained as a challenge.
	
	This study takes an initial step toward extending projection-based descent algorithms to nonsmooth, nonconvex optimization problems. To this end, we concentrate on the minimization problem
	\begin{equation}\label{Main-Intro}
		\min \,\, f(\bm x) \quad \text{s.t.} \quad \bm x\in C,
	\end{equation}
	where $f:\mathbb{R}^n\to\mathbb{R}$ is weakly semismooth and $C\subset\mathbb{R}^n$ is a nonempty, closed, and convex polyhedral set.
	At each iteration of the method, we consider an inner approximation of the Clarke $\varepsilon$-subdifferential of $f$ at the current point over the feasible region $C$, whose least-norm element provides a search direction.
	Since the  search directions are employed within a projected line search procedure, they are not required to be feasible.
	The algorithm combines two complementary line search procedures. An exponential limited backtracking line search is first used to check an Armijo-type sufficient decrease condition by projecting each trial point onto the feasible region. If this condition is satisfied, the algorithm takes a serious step.
	Otherwise, the method admits a null step, and a projected variant of Mifflin's line search \cite{kiwielbook} is invoked to compute a new nonredundant subgradient, which enriches the current approximation of the Clarke $\varepsilon$-subdifferential and enables the computation of a new search direction. 
	We prove that, whenever the objective function $f$ is weakly semismooth and the feasible region $C$ is a polyhedral set, the proposed projected variant of Mifflin's line search terminates after finitely many iterations with a nonredundant subgradient. We characterize the stationary points of problem~\eqref{Main-Intro} through the projection mapping, providing the proposed method with a computable optimality measure based on a necessary optimality condition. We establish the convergence properties of the proposed method under mild assumptions. In particular, when the algorithm generates infinitely many serious steps, we show that two subsequences of the generated iterates can be identified such that every cluster point of each subsequence is stationary. On the other hand, when only finitely many serious steps are generated, we prove that the iterate associated with the final serious step is stationary.

	%The global convergence properties of the proposed method are thoroughly studied, and we prove that every accumulation point of the generated sequence is a stationary point of problem~\eqref{Main-Intro}.
	
	Through  numerical experiments, we assess the practical performance of the proposed method. To this end, we first consider a collection of linearly constrained test problems to illustrate the main features of the proposed method. We also propose and apply a heuristic approach for handling smooth nonlinear constraints. At each iteration, this sequential approach retains the original objective function while replacing each nonlinear constraint with a linear approximation within a trust region. The applicability of the proposed method to large-scale optimization is demonstrated through an image processing problem. Furthermore, we apply the method to approximate the Pareto front of a nonsmooth, nonconvex multiobjective optimization problem. Finally, we consider a data clustering problem in which the cluster centroids are constrained to lie in a polyhedral set.
	
	The remainder of the paper is organized as follows. Section~\ref{Basic-Concepts} introduces the basic concepts and notation. Section~\ref{P-Projection} discusses projection onto polyhedral sets. Section~\ref{Stat-Proj} characterizes stationary points using the projection operator. Section~\ref{Deivation} presents the proposed method in detail. Section~\ref{Conv-Analysis} is devoted to the convergence analysis of the method.  Section~\ref{Numeric} reports the results of the numerical experiments, and Section~\ref{Conclusion} concludes the paper.

	\section{Basic Concepts}\label{Basic-Concepts}
	We denote  by $\mathbb{R}^n$ the $n$-dimensional Euclidean space, and the inner product of any two column vectors $\bm x, \bm y\in\mathbb{R}^n$ is given by $\bm x^T\bm y:=\sum_{i=1}^{n} x_i y_i$, which induces the Euclidean norm $\lVert \bm x\rVert:=\sqrt{\bm x^T \bm x}$. In addition, $\mathcal{B}_\varepsilon(\bm x)$ is the closed  ball  centered at $\bm x\in\mathbb{R}^n$ with radius $\varepsilon\geq0$, i.e.,
	$\mathcal{B}_\varepsilon(\bm x):=\{\bm y\in\mathbb{R}^n \,:\, \lVert\bm y-\bm x \rVert\leq \varepsilon\}$. Moreover, the  infinity norm of the vector $\bm x\in\mathbb{R}^n$ is given by $$\lVert \bm x \rVert_\infty:=\max\big\{\lvert x_i \rvert \,:\, i\in\{1,2,\ldots, n\} \big\}.$$
	Furthermore,
	$\mathbb{R}_+:=(0,+\infty)$, and $\mathbb{N}_0:=\mathbb{N}\cup\{0\}$.
	
	The classical directional derivative of a function $F:\mathbb{R}^n\to\mathbb{R}^m$ at a point $\bm x\in\mathbb{R}^n$ and direction $\bm d\in\mathbb{R}^n$ is given by \cite{Zowe-book}
	\begin{equation}
		F'(\bm x; \bm d):=\lim_{t\downarrow 0}\frac{F(\bm x+t\bm d)-F(\bm x)}{t}.
	\end{equation} 
	If $F'(\bm x, \bm d)$ exists for every  $\bm x, \bm d\in\mathbb{R}^n$, then $F$ is called a directionally differentiable function.

	For a locally Lipschitz function $F:\mathbb{R}^n\to\mathbb{R}^m$, let
	$$\Omega_F:=\left\{\bm x\in\mathbb{R}^n \,:\, F\,\, \text{is not differentiable at}\, \bm x \right\}.$$
	Then, by Rademacher's theorem \cite{Evans2015}, $\mathbb{R}^n\setminus\Omega_F$ is a full measure subset of $\mathbb{R}^n$.
	
	Suppose $f:\mathbb{R}^n\to\mathbb{R}$ is a locally Lipschitz function. The Clarke subdifferential  of $f$ at a given point $\bm x\in\mathbb{R}^n$ is given by \cite{Clarke:Functional_analysis}
	\begin{equation*}
		\partial f(\bm x):=\texttt{conv} \left\{\boldsymbol{\xi}\in\mathbb{R}^n \,:\, \exists \{\bm x_k\}\subset \mathbb{R}^n \setminus\Omega_f \,\,\, \text{s.t.} \,\,\, \bm x_k\to\bm x \,\,\text{and} \,\, \nabla f(\bm x_k)\to \boldsymbol{\xi}   \right\},
	\end{equation*}
	where \texttt{conv} denotes the convex hull of a set. Each element of the set $\partial f(\bm x)$ is called a subgradient.  For any $\varepsilon\geq0$, the Clarke $\varepsilon$-subdifferential of $f$ at a point $\bm x\in\mathbb{R}^n$ is defined as \cite{Makela_book}
	\begin{equation*}
		\partial_\varepsilon f(\bm x):=\texttt{conv}\left\{ \partial f(\bm y) \,\,:\,\, \bm y\in\mathcal{B}_\varepsilon(\bm x)   \right\}.
	\end{equation*}
	Notice that $\partial f(\bm x)=\partial_0 f(\bm x)$, for all $\bm x\in\mathbb{R}^n$. Furthermore, for any $\varepsilon\geq 0$ and $\bm x\in\mathbb{R}^n$, the set $\partial_\varepsilon f(\bm x)$
	is a nonempty, convex, and compact subset of $\mathbb{R}^n$. In addition, for any $\varepsilon\geq 0$, the set-valued map $\partial_\varepsilon f:\mathbb{R}^n\rightrightarrows\mathbb{R}^n$ is locally bounded and upper semicontinuous \cite{Makela_book}.
	
	For the locally Lipschitz vector-valued function $F:\mathbb{R}^n\to\mathbb{R}^m$, the generalized Jacobian at a given point $\bm x\in\mathbb{R}^n$ is defined as \cite{Zowe-book}
	$$\partial F(\bm x):=\texttt{conv}\left\{\bm V\in\mathbb{R}^{m\times n}\, : \, \exists \{\bm x_k\}\subset \mathbb{R}^n \setminus\Omega_F \,\,\, \text{s.t.} \,\,\, \bm x_k\to\bm x \,\,\text{and} \,\, \mathcal J F(\bm x_k)\to \bm V    \right\},$$
	in which, for any $\bm x\in\mathbb{R}^n \setminus\Omega_F$, $\mathcal J F(\bm x)\in\mathbb{R}^{m\times n}$ is the usual Jacobian matrix. Similar to the subdifferential set, for every $\bm x\in\mathbb{R}^n$, the generalized Jacobian $\partial F(\bm x)$ is a nonempty, convex, and compact subset of $\mathbb{R}^{m\times n}$. Moreover, the set-valued map $\partial F:\mathbb{R}^n\rightrightarrows\mathbb{R}^{n\times m}$ is locally bounded and upper semicontinuous \cite{Zowe-book}.
	
	The Clarke directional derivative of the locally Lipschitz function $f:\mathbb{R}^n\to\mathbb{R}$ at point  $\bm x\in\mathbb{R}^n$ and direction $\bm d\in\mathbb{R}^n$ is given by \cite{Clarke:Functional_analysis}
	\begin{equation}\label{Directional Derivative}
		f^\circ(\bm x; \bm d):=\limsup_{\substack{\bm y\to \bm x\\t\downarrow 0}} \frac{f(\bm y+ t\bm d)-f(\bm y)}{t}.
	\end{equation}
	Since locally Lipschitz functions are locally bounded, $f^\circ(\bm x; \bm d)$ exists, for all $\bm x\in\mathbb{R}^n$ and $\bm d\in\mathbb{R}^n$. It is recalled that $f^\circ(\bm x; \bm d)$ can be expressed as the support functional for the subdifferential set $\partial f(\bm x)$; in other words \cite{Clarke:Functional_analysis}
	\begin{equation}\label{support-functional}
		f^\circ(\bm x; \bm d)=\max \left\{\boldsymbol{\xi}^T \bm d \,:\, \boldsymbol{\xi}\in\partial f(\bm x)  \right\}.
	\end{equation}

	Next, we recall two fundamental classes of functions that play a central role in nonsmooth optimization \cite{Zowe-book,bagirov2020}.
	\begin{definition}\label{weakly-semismooth}
		A function $F:\mathbb{R}^n\to\mathbb{R}^m$ is said to be semismooth at $\bm z\in\mathbb{R}^n$ if it is locally Lipschitz at $\bm z$, and the limit 
		\begin{equation}
			\lim_{\substack{\bm V\in\partial F(\bm z+h \bm d')\\ \bm d'\to \bm d,\, h\downarrow 0    }} \bm V \bm d,
		\end{equation}
		exists, for all $\bm d\in\mathbb{R}^n$. Furthermore, $F$ is called weakly semismooth at $\bm z\in\mathbb{R}^n$ if it is locally Lipschitz at $\bm z$, and the limit 
		\begin{equation}
			\lim_{\substack{\bm V\in\partial F(\bm z+h \bm d)\\  h\downarrow 0    }} \bm V \bm d,
		\end{equation}
		exists, for all $\bm d\in\mathbb{R}^n$. Moreover,  $F$ is called (weakly) semismooth if it is (weakly) semismooth at any $\bm z\in\mathbb{R}^n$.
		
		%and for any $\bm d\in\mathbb{R}^n$ and sequences $\{t_i\}\subset\mathbb{R}_+$ and $\{\bm V_i\}\subset\mathbb{R}^{m\times n}$ satisfying $t_i\downarrow 0$ and $\bm V_i\in\partial F(\bm x+t_i \bm d)$, the sequence	 $\{\bm V_i \bm d\}$ has exactly one limit point. Moreover,  $F$ is called weakly semismooth if it is weakly semismooth at any $\bm x\in\mathbb{R}^n$.	
	\end{definition}
	Clearly, any semismooth function is weakly semismooth. Furthermore, a weakly semismooth function $F:\mathbb{R}^n\to\mathbb{R}^m$ is directionally differentiable, and  \cite{Zowe-book} 
	\begin{equation*}
		F'(\bm z; \bm d)=\lim_{\substack{\bm V\in\partial F(\bm z+h \bm d)\\  h\downarrow 0    }} \bm V \bm d.
	\end{equation*}

	%%%%%%%%%%%%%%%%%%%%%%%%%%%%%%%%%%%%%%%%%%%%%%%%%%%%%
	%A function $f:\mathbb{R}^n\to\mathbb{R}$ is called weakly upper semismooth if for any $\bm x, \bm d\in\mathbb{R}^n$ and sequences $\{h_i\}\subset\mathbb{R}_+$ and $\{\boldsymbol{\xi}_i \}\subset \mathbb{R}^n$ satisfying $t_i\downarrow 0$ as $i\to\infty$ and $\boldsymbol{\xi}_i\in\partial f(\bm x+ t_i \bm d)$ one has
	%\begin{equation}
	%\liminf_{i\to\infty} \frac{f(\bm x+t_i \bm d)-f(\bm x)}{t_i}\leq \limsup_{i\to\infty} \boldsymbol{\xi}_i^T \bm d.
	%\end{equation}
	%%%%%%%%%%%%%%%%%%%%%%%%%%%%%%%%%%%%%%%%%%%%%%%%%%%%%%
	
	%As proved in \cite{Ulbrich-semismooth}, a directionally differentiable function $F:\mathbb{R}^n\to\mathbb{R}^m$ is semismooth at $\bm x\in\mathbb{R}^n$ if for the sequences $\{\bm h_i\}\subset\mathbb{R}^n$ and $\{\bm V_i\}\subset\mathbb{R}^{m\times n}$ satisfying $\bm h_i\to \bm 0$ and $\bm V_i\in\partial F(\bm x +\bm h_i)$, one has 
	%\begin{equation}\label{Suff-cond-semi}
	%\lVert F(\bm x+\bm h_i)-F(\bm x)-\bm V_i \bm h_i\rVert=o(\lVert \bm h_i\rVert),
	%\end{equation}
	%where $o(\lVert \bm h_i\rVert)$ is a sequence satisfying  $o(\lVert \bm h_i\rVert)/\lVert \bm h_i\lVert\to \bm 0$, as $i\to\infty$.
	
	\section{ Projection onto Polyhedral Sets}\label{P-Projection}
	Let $C\subseteq\mathbb{R}^n$ be a nonempty, closed, and convex subset of $\mathbb{R}^n$. The orthogonal projection mapping $P_C:\mathbb{R}^n\to\mathbb{R}^n$ is defined by
	\begin{equation}\label{Projection_Problem}
		P_C(\bm x):=\text{argmin} \, \left\{\lVert \bm y-\bm x\lVert \,\,:\,\, \bm y\in C    \right\}.
	\end{equation}
	Due to the closedness and convexity of the nonempty set $C$, for any $\bm x\in\mathbb{R}^n$, $P_C(\bm x)$ exists and is uniquely determined. Moreover, $P_C$ is  nonexpansive \cite{Amir_Beck_Nonlinear}, i.e.,
	\begin{equation}
		\lVert P_C(\bm x)-P_C(\bm y)\rVert \leq \lVert \bm x-\bm y \rVert, \quad \forall \, \bm x, \bm y\in\mathbb{R}^n,
	\end{equation} 
	and hence it is a continuous function over $\mathbb{R}^n$. For a given $\bm x\in\mathbb{R}^n$, it is essentially well-known that $\bm y=P_C(\bm x)$ if and only if \cite{Amir_Beck_Nonlinear}
	\begin{equation}\label{First-Pro-Theorem}
		(\bm z-\bm y)^T(\bm x-\bm y)\leq 0, \quad \forall \, \bm z\in C.
	\end{equation}
	%As another important feature of the orthogonal projection map, let $\bm x, \bm d\in\mathbb{R}^n$, and $M_1\geq M_2>0$. Then \cite{Amir_Beck_Nonlinear}
	%\begin{equation}\label{projection-feature2}
	%\frac{\left \lVert M_2 \left[ \bm x -P_C\left(\bm x -\frac{1}{M_2} \bm d\right)   \right]   \right\rVert}{M_2} \geq \frac{\left \lVert M_1 \left[ \bm x -P_C(\bm x -\frac{1}{M_1} \bm d)   \right]   \right\rVert}{M_1}.
	%\end{equation}
	
	In what follows, we consider $C$ as a nonempty, closed, and convex \emph{polyhedral} set, i.e., 
	\begin{equation}\label{Polyhedral_Set}
		C:=\{ \bm y\in\mathbb{R}^n \,\,:\,\,\bm  A\bm y\leq \bm b  \}\neq \emptyset,
	\end{equation}
	in which $\bm A\in\mathbb{R}^{m\times n}$ and $\bm b\in\mathbb{R}^m$. In this situation, it is proved in \cite{Zowe-book} that  $P_C$ is a directionally differentiable function. In addition to this, the following lemma reveals that $P_C$ is a piecewise linear function, which is a consequence of Karush-Kuhn-Tucker (KKT) optimality conditions. Similar versions of this result can be found in \cite{Projection-Poly, Amir_Beck_Nonlinear}. However, the proof presented here includes some technical details that are essential for the development of the subsequent results.

	%in what follows, we shall prove that $P_C$ is indeed a semismooth map. To this end, we need the following auxiliary result.
	
	\begin{lemma}\label{L1}
		Suppose that $C\subseteq\mathbb{R}^n$ is a nonempty polyhedral set given by \eqref{Polyhedral_Set}. Then, there exist $N\in\mathbb{N}$, and  subsets $R_j$ of $\mathbb{R}^n$, $j=1,2,\ldots, N$, such that $\mathbb{R}^n=\cup_{j=1}^{N} R_j$ and
		\begin{equation}
			P_C(\bm x)=\bm M_j \bm x + \bm c_j, \quad \forall\,  \bm x\in R_j,
		\end{equation}	
		in which,  $\bm M_j\in\mathbb{R}^{n\times n}$ and $\bm c_j\in\mathbb{R}^n$, for $j=1,\ldots,N$.
	\end{lemma}
	\begin{proof}
		Since the objective function of problem~\eqref{Projection_Problem} is convex and $C$ is a polyhedral set, KKT  conditions are indeed necessary and sufficient optimality conditions. For a given $\bm x\in\mathbb{R}^n$, suppose 
		$$\bm y=P_C(\bm x)=\text{argmin} \, \left\{\frac{1}{2}\lVert \bm y-\bm x\lVert^2 \,\,:\,\, \bm y\in C    \right\}.$$
		Then, KKT conditions imply the existence of nonnegative vector of Lagrange multipliers $\boldsymbol{\lambda}:=(\lambda_1,\ldots,\lambda_m)^T$ such that
		\begin{align}\label{KKT-Cond}
			\bm y-\bm x+\bm A^T \boldsymbol{\lambda} =\bm 0 \quad \text{and} \quad \lambda_i(\bm a^i\bm y-b_i)=0, \,\, \forall\, i\in\{1,\ldots,m\},
		\end{align}
		in which, $\bm a^i$ and $b_i$ denote the $i$-th row of $\bm A\in\mathbb{R}^{m\times n}$ and $i$-th component of $\bm b\in\mathbb{R}^m$, respectively.	Define the optimal active set $\mathrm{I}:=\left\{i\in\{1,\ldots,m\} \,:\, \bm a^i \bm y=b_i   \right\}$.
		
		First,  assume $\mathrm{I}\neq \emptyset$. Then, $\lambda_i=0$ for any $i\notin\mathrm{I}$, and
		\begin{equation}\label{KKT-Active}
			\bm y-\bm x+\bm A^T_{\mathrm{I}} \boldsymbol{\lambda}_{\mathrm{I}} =\bm 0 \quad \text{and} \quad \bm A_{\mathrm{I}}\bm y=\bm b_{\mathrm{I}},
		\end{equation}
		in which, $\bm A_{\mathrm{I}}:=[\bm a^i]_{i\in\mathrm{I}}$,  $\bm b_{\mathrm{I}}:=(b_i)_{i\in\mathrm{I}}$, and $\boldsymbol{\lambda}_{\mathrm{I}}:=(\lambda_i)_{i\in\mathrm{I}}$. With no loss of generality, one may assume that $\bm A_{\mathrm{I}}$ is a full row rank matrix. Next, 
		it follows from~\eqref{KKT-Active} that
		\begin{equation*}
			\bm y=\bm x- \bm A_{\mathrm{I}}^T \left(\bm A_{\mathrm{I}} \bm A_{\mathrm{I}}^T \right)^{-1} \left(\bm A_{\mathrm{I}} \bm x-\bm b_{\mathrm{I}}   \right),
		\end{equation*} 
		and hence, one can write $\bm y=\bm M \bm x + \bm c$ such that
		\begin{equation*}\label{M-c-1'}
			\bm M:= \bm I_n- \bm A_{\mathrm{I}}^T \left(\bm A_{\mathrm{I}} \bm A_{\mathrm{I}}^T \right)^{-1} \bm A_{\mathrm{I}}\in\mathbb{R}^{n\times n} \quad \text{and} \quad \bm c:=\bm A_{\mathrm{I}}^T \left(\bm A_{\mathrm{I}} \bm A_{\mathrm{I}}^T \right)^{-1} \bm b_{\mathrm{I}}\in\mathbb{R}^n,
		\end{equation*}
		where, $\bm I_n\in\mathbb{R}^{n\times n}$ is the identity matrix.
		
		In case $\mathrm{I}=\emptyset$, it immediately follows from \eqref{KKT-Cond} that $\bm y=\bm x$, and it is sufficient to set 
		$\bm M:= \bm I_n\in\mathbb{R}^{n\times n} $ and $\bm c:=\bm 0 \in\mathbb{R}^n$ to see that $\bm y=\bm M\bm x+\bm c$.
		
		Therefore, by setting $$R:=\{\bm x\in\mathbb{R}^n :  \text{the optimal active set corresponding to}\, \bm x \,\, \text{is}\,\, \mathrm{I}  \}\subseteq\mathbb{R}^n,$$
		we conclude the existence of $\bm M\in\mathbb{R}^{n\times n}$ and $\bm c\in\mathbb{R}^n$ such that $\bm y=\bm M \bm x + \bm c$, for all $\bm x\in R$.
		
		Next, let $\mu:=\{1,2,\ldots,m\}$ and $\mathcal{P}(\mu):=\{\mathrm{I}^1, \mathrm{I}^2,\ldots,\mathrm{I}^N  \}$ denote the power set of the set  $\mu$. For any $j\in\{1,\ldots N\}$, define
		\begin{equation*}
			R_j:=\{\bm x\in\mathbb{R}^n :  \text{the optimal active set corresponding to}\, \bm x \,\, \text{is}\,\, \mathrm{I}^j  \}.
		\end{equation*}
		Then, it is evident that $\mathbb{R}^n=\cup_{j=1}^{N} R_j$. Moreover, for any $j\in\{1,\ldots N\}$,  the above arguments ensure the existence of  $\bm M_j\in\mathbb{R}^{n\times n}$ and $\bm c_j\in\mathbb{R}^n$ such that
		\begin{equation*}
			\bm y=P_C(\bm x)=\bm M_j \bm x + \bm c_j, \quad \forall \, \bm x\in R_j.
		\end{equation*}
	\end{proof}

	\begin{corollary}\label{Corollary1}
		Suppose that $C\subseteq\mathbb{R}^n$ is a nonempty polyhedral set given by \eqref{Polyhedral_Set}. For any  $j\in\{1,2,\ldots,N\}$,	let $R_j\subseteq\mathbb{R}^n, \bm M_j\in\mathbb{R}^{n\times n}$, and $\bm c_j\in\mathbb{R}^n$ be as obtained in Lemma~\ref{L1}. Then,
		\begin{itemize}
			\item [(i)] For any $\bm x\in\texttt{cl}\, R_j$, we have $	P_C(\bm x)= \bm M_j\bm x + \bm c_j$, where \texttt{cl} denotes the closure of a set.
			\item[(ii)] For each $j\in\{1,2,\ldots,N\}$, $R_j$ is a convex subset of $\mathbb{R}^n$.
		\end{itemize}

	\end{corollary}

	\begin{proof}
		(i)	This is an immediate consequence of Lemma~\ref{L1} and  continuity of the map $P_C:\mathbb{R}^n\to\mathbb{R}^n$.
		
		(ii) Assume $ j\in\{1,2,\ldots,N\}$ is arbitrary and $\mu:=\{1,2,\ldots,m\}$. For the sake of simplicity in notations, let $R:=R_j$ and  $\mathrm{I}:=\mathrm{I}^j$. First, it is assumed that $\mathrm{I}\neq\emptyset$.  Suppose $\bm x_1, \bm x_2\in R$ are arbitrary, and $t\in[0, 1]$ is given. In addition, suppose $\bm y_1:=P_C(\bm x_1)$, $\bm y_2:=P_C(\bm x_2)$, and $\bm x_t:=t \bm x_1 + (1-t)\bm x_2$. We need to show that $\bm x_t\in R$. Since $\bm x_1, \bm x_2\in R$, we have
		\begin{equation}\label{Co-1-1}
			\bm A_{\mathrm{I}} \bm y_1 = \bm b_{\mathrm{I}} \quad \text{and} \quad \bm A_{\mu\setminus\mathrm{I}} \bm y_1 < \bm b_{\mu\setminus\mathrm{I}},
		\end{equation}
		and
		\begin{equation}\label{Co-1-2}
			\bm A_{\mathrm{I}} \bm y_2 = \bm b_{\mathrm{I}} \quad \text{and} \quad \bm A_{\mu\setminus\mathrm{I}} \bm y_2< \bm b_{\mu\setminus\mathrm{I}}.
		\end{equation}
		As $\bm y_1=P_C(\bm x_1)$ and $\bm y_2=P_C(\bm x_2)$ with $\mathrm{I}$ as the optimal active set,  KKT optimality conditions ensure the existence of nonnegative vectors of Lagrange multipliers $\boldsymbol{\lambda}^1, \boldsymbol{\lambda}^2\in\mathbb{R}^m$ such that 
		\begin{equation}\label{Co-1-3}
			\bm y_1 - \bm x_1 + \bm A^T_{\mathrm{I}} \boldsymbol{\lambda}^1_{\mathrm{I}} =\bm 0 \quad \text{and} \quad \bm y_2 - \bm x_2  +\bm A^T_{\mathrm{I}} \boldsymbol{\lambda}^2_{\mathrm{I}} =\bm 0. 
		\end{equation}
		Let $\bm y_t:= t \bm y_1 + (1-t) \bm y_2$. Then, by multiplying the first and second equations in \eqref{Co-1-3} by $t$ and $1-t$, respectively, and then summing them up, we obtain
		\begin{equation}\label{Co-1-4}
			\bm y_t - \bm x_t + \bm A^T_{\mathrm{I}} \big(t \boldsymbol{\lambda}^1_{\mathrm{I}} + (1-t)\boldsymbol{\lambda}^2_{\mathrm{I}}  \big) = \bm 0.
		\end{equation}
		Clearly, $t \boldsymbol{\lambda}^1_{\mathrm{I}} + (1-t)\boldsymbol{\lambda}^2_{\mathrm{I}}\geq \bm 0 $. Moreover, in view of \eqref{Co-1-1} and \eqref{Co-1-2}, one can write
		\begin{equation}\label{Co-1-5}
			\bm A_{\mathrm{I}} \bm y_t=\bm b_{\mathrm{I}} \quad \text{and} \quad \bm A_{\mu\setminus\mathrm{I}} \bm y_t<\bm b_{\mu\setminus\mathrm{I}}.
		\end{equation}
		Next, it follows from \eqref{Co-1-4} and \eqref{Co-1-5} that $\bm y_t=P_C(\bm x_t)$ with $\mathrm{I}$ as the optimal active set. Consequently, $\bm x_t\in R$, which means $R$ is convex.
		
		In case $\mathrm{I}=\emptyset$, it is easy to see that $R=\left\{\bm x\in\mathbb{R}^n\, : \, \bm A \bm x < \bm b   \right\}$, which is clearly a convex set.
	\end{proof}

	It is shown in \cite{Ulbrich-semismooth} that every piecewise smooth function is semismooth. Thus, as a consequence of Lemma~\ref{L1},
	the projection mapping
	$P_C:\mathbb{R}^n\to\mathbb{R}^n$ is semismooth whenever the set $C\subseteq\mathbb{R}^n$ is given by \eqref{Polyhedral_Set}. Moreover,
	Mifflin establishes that the class of semismooth functions is closed under composition \cite{Mifflin}. In addition, it is proved in \cite{Zowe-book} that the composition of a continuously differentiable function with a weakly semismooth function remains weakly semismooth.
	In the following, for a polyhedral set $C\subseteq\mathbb{R}^n$, we show that the composition of a weakly semismooth function with the projection mapping is again weakly semismooth. To this end, we first establish the following auxiliary result.

	\begin{lemma}\label{L2'}
		Suppose $\bm z, \bm d\in\mathbb{R}^n$, and the sequence $h_k\downarrow 0$ is given.	Let $R_j\subseteq\mathbb{R}^n, j=1,2,\ldots,N$, be as derived in Lemma~\ref{L1}. Then, there exist $\bar{k}\in\mathbb{N}$ and $\bar{j}\in\{1,2,\ldots,N\}$ such that
		\begin{equation}\label{L1'-1}
			\bm z+ h_k \bm d \in R_{\bar j}, \quad \forall \, k\geq \bar{k}.
		\end{equation}
		
	\end{lemma}
	\begin{proof}
		In case $\bm d=\bm 0$, the assertion follows immediately from the fact that $\mathbb{R}^n=\cup_{j=1}^{N} R_j$. Let $\bm d\neq \bm 0$, and
		suppose by indirect proof that the assertion does not hold. Then, for any $\bar k\in\mathbb{N}$, there exist  $k_3>k_2>k_1\geq \bar k$ along with $\hat{j}\in \{1,2,\ldots,N\}$ such that 
		\begin{equation*}
			\bm z+ h_{k_1} \bm d\in R_{\hat{j}}, \quad \bm z+ h_{k_2} \bm d\notin R_{\hat{j}}, \,\, \text{and} \,\,\, \bm z+ h_{k_3} \bm d\in R_{\hat{j}}. 
		\end{equation*}
		Since $k_3>k_2>k_1$, this means $R_{\hat{j}}$ is not a convex set, a contradiction. 
	\end{proof}
	
	\begin{theorem}\label{T0}
		Assume $C\subseteq\mathbb{R}^n$ is a nonempty polyhedral set given by \eqref{Polyhedral_Set}, and $f:\mathbb{R}^n\to\mathbb{R}$ is weakly semismooth. Then, $f\circ P_C:\mathbb{R}^n\to\mathbb{R}$ is weakly semismooth.
	\end{theorem}
	\begin{proof}
		Suppose $\bm z, \bm d\in\mathbb{R}^n$ are fixed, and the sequence $h_k\downarrow 0$ is arbitrary. For any $k\in\mathbb{N}$, let $\bm z_k:=\bm z + h_k \bm d$, and $\boldsymbol{\xi}_k\in\partial \big(f\circ P_C\big)(\bm z_k) $.  By using the chain rule, we have
		\begin{align*}
			\partial \big(f\circ P_C\big)(\bm z_k)&\subseteq \texttt{conv} \left\{\boldsymbol{\xi}\in\mathbb{R}^n \,: \, \boldsymbol{\xi}=\bm G^T \bm w,  \bm G\in \partial P_C(\bm z_k) \, \text{and} \, \bm w\in \partial f\big(P_C(\bm z_k) \big)   \right\}\\&=:E(\bm z_k).
		\end{align*}
		As $E(\bm z_k)$ is a convex compact subset of $\mathbb{R}^n$, for each $\boldsymbol{\xi}_k\in E(\bm z_k)$, there exist $\hat{\boldsymbol{\xi}}_k, \bar{\boldsymbol{\xi}}_k\in E(\bm z_k)$  such that
		\begin{align}\label{T0-1}
			\bar{\boldsymbol{\xi}}_k^T \bm d:= \inf_{\boldsymbol{\xi}\in E(\bm z_k)} \boldsymbol{\xi}^T \bm d \leq \boldsymbol{\xi}_k^T \bm d\leq \sup_{\boldsymbol{\xi}\in E(\bm z_k)} \boldsymbol{\xi}^T \bm d:= 	\hat{\boldsymbol{\xi}}_k^T \bm d.
		\end{align}
		Since $\hat{\boldsymbol{\xi}}_k, \bar{\boldsymbol{\xi}}_k\in E(\bm z_k)$, there exist $\bar{\bm G}_k, \hat{\bm G}_k\in\partial P_C(\bm z_k)$ and $\bar{\bm w}_k, \hat{\bm w}_k\in\partial f\big(P_C(\bm z_k)   \big)$ such that
		\begin{equation}\label{T0-2}
			\hat{\boldsymbol{\xi}}_k= \hat{\bm G}_k^T \hat{\bm w}_k \quad \text{and} \quad \bar{\boldsymbol{\xi}}_k= \bar{\bm G}_k^T \bar{\bm w}_k. 
		\end{equation}
		Moreover, $\bm z_k\to \bm z $ as $k\to\infty$,  and hence the local boundedness of the subdifferential map together with its upper semicontinuity  implies the existence of $\mathcal{K}\subseteq\mathbb{N}$ along with $\bar{\bm G}, \hat{\bm G}\in\partial P_C(\bm z)$ and $\bar{\bm w}, \hat{\bm w}\in\partial f\big(P_C(\bm z)   \big)$ such that
		\begin{equation}\label{T0-3}
			\bar{\bm G}_k\to \bar{\bm G}, \quad \hat{\bm G}_k\to \hat{\bm G}, \quad \bar{\bm w}_k\to \bar{\bm w}, \quad \text{and} \quad \hat{\bm w}_k\to\hat{\bm w},
		\end{equation}
		as $k\xrightarrow{k\in\mathcal{K}}\infty$. On the other hand, by the semismoothness of the map $P_C$, we have 
		\begin{equation}\label{T0-4}
			\left(\hat{\bm G}_k^T \hat{\bm w}_k\right)^T \bm d \to \hat{\bm w}^T P_C'(\bm z; \bm d) \quad \text{and} \quad   \left(\bar{\bm G}_k^T \bar{\bm w}_k\right)^T \bm d \to \bar{\bm w}^T P_C'(\bm z; \bm d),
		\end{equation} 
		as $k\xrightarrow{k\in\mathcal{K}}\infty$. Next, in view of Lemma~\ref{L2'},  there exists $\bar{j}\in\{1,2,\ldots, N\}$ and $\bar k\in\mathbb{N}$ such that $\bm z_k\in R_{\bar{j}}$, for all $k\geq\bar k$. Hence,  Lemma~\ref{L1} implies   $P_C(\bm z_k)=\bm M_{\bar{j}} \bm z_k+\bm c_{\bar{j}}$, for all $k\geq\bar k$. In addition, $\bm z\in \texttt{cl} R_{\bar{j}}$, and part (i) of Corollary~\ref{Corollary1} gives $P_C(\bm z)=\bm M_{\bar{j}} \bm z+\bm c_{\bar{j}}$. Consequently,
		\begin{equation}\label{T0-5}
			P_C'(\bm z; \bm d)=\lim_{k\to\infty} \frac{P_C(\bm z_k)-P_C(\bm z)}{h_k}= \bm M_{\bar{j}} \bm d.
		\end{equation}
		Now, we conclude from \eqref{T0-1}, \eqref{T0-2}, \eqref{T0-3}, \eqref{T0-4}, and \eqref{T0-5} that
		\begin{equation}\label{T0-6}
			\bar{\bm w}^T \bm M_{\bar{j}} \bm d \leq \liminf_{k\to\infty} \boldsymbol{\xi}_k^T \bm d \leq \limsup_{k\to\infty} \boldsymbol{\xi}_k^T \bm d \leq 	\hat{\bm w}^T \bm M_{\bar{j}} \bm d .
		\end{equation}
		Next, we show that    $\lim_{k\to \infty} \boldsymbol{\xi}_k^T \bm d$ exists. To this end, in virtue of \eqref{T0-6}, it is sufficient to prove  $\hat{\bm w}^T \bm M_{\bar{j}} \bm d=\bar{\bm w}^T \bm M_{\bar{j}} \bm d $. Notice that
		\begin{equation*}
			P_C(\bm z_k)= \bm M_{\bar{j}} (\bm z+h_k \bm d) + \bm c_{\bar{j}}= P_C(\bm z) + h_k \bm M_{\bar{j}} \bm d, \quad \forall \, k\geq \bar{k}.
		\end{equation*} 
		Thus, $\bar{\bm w}_k, \hat{\bm w}_k\in\partial f\big(P_C(\bm z_k)\big)=\partial f\big(P_C(\bm z)+h_k\bm M_{\bar{j}} \bm d\big)$, for all $k\geq \bar{k}$. Therefore, weakly semismoothness of the function $f$ implies that the sequences $\{\bar{\bm w}_k^T \bm M_{\bar{j}} \bm d   \}$ and $\{\hat{\bm w}_k^T \bm M_{\bar{j}} \bm d   \}$ have exactly the same limit, which, in turn, are $\bar{\bm w}^T \bm M_{\bar{j}} \bm d$ and $\hat{\bm w}^T \bm M_{\bar{j}} \bm d$. Consequently,
		$$\bar{\bm w}^T \bm M_{\bar{j}} \bm d=\hat{\bm w}^T \bm M_{\bar{j}} \bm d.$$
		Eventually, since the limit of the sequence $\{ \boldsymbol{\xi}_k^T \bm d\}$ was independent of our choice of $\boldsymbol{\xi}_k\in\partial\big(f\circ P_C\big)(\bm z_k)$, the proof is complete.
	\end{proof}

	\section{ Stationary Points and the Projection Operator}\label{Stat-Proj}
	We now return to the main problem 
	\begin{equation}\label{Main_Problem-1}
		\min \, f(\bm x) \quad \text{s.t.} \quad \bm x\in C,
	\end{equation}
	where $f:\mathbb{R}^n\to\mathbb{R}$ is a weakly semismooth function, and $C\subseteq\mathbb{R}^n$ is a nonempty  polyhedral set given by \eqref{Polyhedral_Set}. For a point $\bm x^*\in C$ to be a local minimizer of problem~\eqref{Main_Problem-1}, it is necessary that \cite{Clarke1990} 
	\begin{equation}\label{stationarity-condition}
		f^\circ(\bm x^*; \bm x-\bm x^*)\geq 0, \quad \forall\, \bm x\in C.
	\end{equation}
	The point $\bm x^*\in C$ that satisfies the above condition is called a \emph{stationary} point. In the following theorem, we characterize stationary points of problem~\eqref{Main_Problem-1} based on the projection operator. Before it, let $\mathcal{T}_C(\bm x)$ and $\mathcal{N}_C(\bm x)$ denote the tangent and normal cones to the convex set $C$ at the point $\bm x\in C$, respectively \cite{Rockafellar2004}.
	
	\begin{theorem}\label{T2}
		A point $\bm x^*\in C$ is a stationary point for problem~\eqref{Main_Problem-1} if and only if there exist $\boldsymbol{\xi}^*\in\partial f(\bm x^*)$ and $t>0$ such that
		\begin{equation}\label{projection-stationary}
			\bm x^*=P_C(\bm x^*-t\boldsymbol{\xi}^*).
		\end{equation} 
		
	\end{theorem}
	\begin{proof}
		Assume first that, for some $\boldsymbol{\xi}^*\in\partial f(\bm x^*)$ and $t>0$, we have
		$$\bm x^*=P_C(\bm x^*-t\boldsymbol{\xi}^*).$$
		Thus, in virtue of~\eqref{First-Pro-Theorem}, one can write
		\begin{equation*}
			{\boldsymbol{\xi}^*}^T(\bm x-\bm x^*)\geq 0, \quad \forall\, \bm x\in C,
		\end{equation*}
		which means
		$$f^\circ(\bm x^*; \bm x-\bm x^*)=	\max \left\{\boldsymbol{\xi}^T (\bm x-\bm x^*) \,:\, \boldsymbol{\xi}\in\partial f(\bm x^*)  \right\}\geq 0, \quad \forall\, \bm x\in C,$$
		yielding the stationarity of the point $\bm x^*$.
		
		Next, suppose $\bm x^*$ is a stationary point for problem~\eqref{Main_Problem-1}. Since $C\subseteq\mathbb{R}^n$ is a convex set, we know that stationarity condition~\eqref{stationarity-condition} is equivalent to $$f^\circ(\bm x^*; \bm d)\geq 0, \quad \forall \, \bm d\in\mathcal{T}_C(\bm x^*), $$
		and hence $\bm 0\in\partial f(\bm x^*)+\mathcal{N}_C(\bm x^*)$. Therefore, there exists $\boldsymbol{\xi}^*\in\partial f(\bm x^*)$ such that $-\boldsymbol{\xi}^*\in\mathcal{N}_C(\bm x^*)$. Thus,
		\begin{equation*}
			{\boldsymbol{\xi}^*}^T(\bm x-\bm x^*)\geq 0, \quad \forall \, \bm x\in C.
		\end{equation*}
		In view of~\eqref{First-Pro-Theorem}, the latter inequality is equivalent to the fact that $$\bm x^*=P_C(\bm x^*-t\boldsymbol{\xi}^*), \quad \forall \, t>0,$$ 
		which completes the proof.
	\end{proof}

	\section{ Derivation of the Algorithm}\label{Deivation}
	The main aim of this section is to develop a descent iterative algorithm in order to find a stationary point for problem~\eqref{Main_Problem-1}. To this end, in the light of Theorem~\ref{T2}, we need to find a point $\bm x^*\in C$ which satisfies condition~\eqref{projection-stationary}.

	In nonsmooth optimization, it is well known that an effective descent direction for a function $f:\mathbb{R}^n\to\mathbb{R}$ at $\bm{x}\in\mathbb{R}^n$ can be obtained by computing the least-norm element of the Clarke $\varepsilon$-subdifferential of $f$ at $\bm{x}$. More precisely, for some $\varepsilon>0$, let $\boldsymbol{\xi}^*\neq\bm{0}$ be an optimal solution of the following minimization problem:
	\begin{equation}\label{subproblem-1}
		\min \left\{ \lVert\boldsymbol{\xi}\rVert : \boldsymbol{\xi}\in\partial_\varepsilon f(\bm{x}) \right\}.
	\end{equation}
	Then, the vector $\bm{d}:=-\boldsymbol{\xi}^*$ provides an effective descent direction for $f$ at $\bm{x}$, and is sometimes referred to as the $\varepsilon$-steepest descent direction \cite{Burke2005}. However, computing the entire $\varepsilon$-subdifferential $\partial_\varepsilon f(\bm{x})$ can be computationally demanding in many practical situations.
	
	Regarding our main minimization problem
	
	$$
	\min \,\, f(\bm{x}) \quad \text{s.t.} \quad \bm{x}\in C\subset \mathbb{R}^n,
	$$
	in order to find an effective search direction for $f$ at $\bm{x}\in C$, we propose an iterative process that \emph{sequentially} improves the current approximation of $\partial_\varepsilon f(\bm{x})$ over the feasible region $C$. 
	We shall show that this process either yields a search direction that produces a projected serious step with a significant reduction in the objective function, or, if it generates an infinite sequence of improvements at the current point, establishes that the current point is stationary.

	% with the aim of taking a projected serious step that yields a significant reduction in the objective function.
	
	% The process is continued until either a projected step along the resulting search direction yields a significant reduction in the objective function or it turns out that the current point is indeed an accurate approximation of a stationary point.

	Throughout this section, parallel to what is standard in smooth optimization,  it is assumed that we have a subroutine that can evaluate $f(\bm x)$ and one arbitrary subgradient $\boldsymbol{\xi}\in\partial f(\bm x)$, at each $\bm x\in C$. 
	
	\subsection{  Taking a Serious Step} \label{subsec1}
	
	Suppose we are at the $k$-th iteration of the method, and $\varepsilon_k\in(0, 1)$ is the current radius of the region within which we collect the subgradient information of the objective function. For the current point $\bm x_k\in C\subseteq\mathbb{R}^n$ and $\alpha_k:=\frac{\varepsilon_k}{\sqrt{n}}>0$, define
	\begin{equation}
		C_k:=C\cap\left\{\bm x\in\mathbb{R}^n \,\,: \,\, \lVert \bm x-\bm x_k \rVert_\infty\leq \alpha_k   \right\}.
	\end{equation}
	Obviously, $\bm x_k\in C_k$, and hence $C_k$ is a nonempty polyhedral set.

	Let $\boldsymbol{\xi}_k\in\partial f(\bm y)$, for some $\bm y\in C_k$, and denote by $G_k$ the current bundle of subgradients satisfying
	\begin{equation}
		\boldsymbol{\xi}_k\in G_k\qquad \text{and} \qquad G_k\subseteq \partial_{\varepsilon_k} f(\bm x_k).
	\end{equation}
	Assume that $\bm g^*_k$ is the least-norm element of $\texttt{conv}\{G_k\}\subseteq\partial_{\varepsilon_k} f(\bm x_k)$. In other words
	\begin{equation}\label{least-norm}
		\bm g^*_k:=\text{argmin} \big\{\lVert \bm g\lVert \,\, :\,\, \bm g\in \texttt{conv}\{G_k\}    \big\}.
	\end{equation}
	Next, consider the map $W:\mathbb{R}_+\times \mathbb{N}_0\to\mathbb{R}^n$ which is given by
	\begin{equation}
		W(t,k):= \frac{1}{t}\big(\bm x_k-P_C(\bm x_k-t\bm g^*_k)\big).
	\end{equation}
	Let us proceed with assuming that, for some parameter $\tilde{t}\in(0,1]$, the norm of $W(\tilde{t}, k)$ is sufficiently large, i.e., for some $\delta_k>0$, we have
	$$\left\lVert  W(\tilde t, k) \right\rVert>\delta_k,$$
	(we shall soon specify a value to the parameter $\tilde{t}$; see \eqref{t-tilde}). Then, $\lVert \bm g^*_k\rVert\neq 0$, and  one can define the normalized search direction $\bm d_k$ by $\bm d_k:=- \bm g^*_k/\rVert \bm g^*_k\rVert$. In case $\texttt{conv}\{G_k\}$ is an adequate inner approximation of the Clarke $\varepsilon_k$-subdifferential of $f$ at $\bm x_k$ over the feasible region $C$, one can employ a limited backtracking line search along the search direction $\bm d_k$ to take a projected serious step.
	Such a line search procedure has been presented in Algorithm~\ref{exp-limited-line-search}.

	\begin{algorithm}
		\caption{An Exponential  Limited Backtracking Line Search}
		\label{exp-limited-line-search}
		\hspace*{\algorithmicindent}\textbf{Inputs:} Radius $\varepsilon_k\in(0,1)$, current point $\bm x_k\in C$, and search direction ${\bm d}_k=-\bm g^*_k/\lVert \bm g^*_k\rVert$ with $\bm g^*_k\neq \bm 0$. \\
		\hspace*{\algorithmicindent}\textbf{Parameters:} Scale parameters  $p\in\mathbb{N}$ and $\hat{t}\in(0, 1)$, maximum number of backtrack steps $S_{\max}\in\mathbb{N}$ such that $S_{\max}>p$, $\tilde t:=[\exp(\frac{\log \hat{t}}{p})]^{S_{\max}} $, and sufficient decrease parameter $\beta\in(0,1)$.  \\
		\hspace*{\algorithmicindent}\textbf{Outputs:} Step size $t_k\geq 0$, and indicator $I_k\in\{ 0, 1, 2\}$.\\
		%\hspace*{\algorithmicindent}\textbf{Requirements:} nothing \\
		
		%\hspace*{\algorithmicindent}\textbf{Function:} $\{\bm s, {\rm \bm I}\}$\,\,=\,\,\texttt{T-PLS}\,($\varepsilon$, $\mathbf x$, ${\bm d}$)
		
		\begin{algorithmic}[1]
			\STATE{\textbf{Initializations:} Set $I_k:=0$ and $t_k:=0$ ;}
			\FOR{$s=0, 1, 2,\ldots,S_{\max}$}
			\STATE{Set $ t:=[\exp(\frac{\log \hat{t}}{p})]^{s}$ ;}
			\IF{$s\neq p\,\, \text{and} \,\,f\big(P_C(\bm x_k + t \bm d_k) \big)-f\left(\bm x_k \right)\leq - \beta t \left\lVert W(\tilde t, k) \right\rVert^2  $ }
			\STATE{ Set $t_k:=t, I_k:=1$, and \textbf{STOP} ;}
			%\STATE{ \textbf{Break} ;  }
			\ENDIF
			\IF{$s= p\,\, \text{and} \,\,f\big(P_{C_k}(\bm x_k + t \bm d_k) \big)-f\left(\bm x_k \right)\leq - \beta t \left\lVert W(\tilde t, k) \right\rVert^2  $ }
			\STATE{ Set $t_k:=t, I_k:=2$, and \textbf{STOP} ;}
			%\STATE{ \textbf{break} ;  }
			\ENDIF
			
			\ENDFOR

		\end{algorithmic}
		%\hspace*{\algorithmicindent}\textbf{ End Function}
	\end{algorithm}
	
	Regarding Algorithm~\ref{exp-limited-line-search}, some explanations are necessary. Since $\hat{t}\in(0, 1)$ and $p\in\mathbb{N}$, we have $\frac{\log \hat{t}}{p}<0$, and hence 
	\begin{equation}\label{decrease-property}
		0<\left[\exp\left(\frac{\log \hat{t}}{p}\right) \right]^{s+1}<\left[\exp\left(\frac{\log \hat{t}}{p}\right) \right]^{s}\leq 1, \quad \forall\,\, s\in\mathbb{N}_0.
	\end{equation}
	Therefore, if Algorithm~\ref{exp-limited-line-search} terminates with $I_k=1$ or $I_k=2$, we have
	\begin{equation}\label{t-tilde}
		1\geq t_k\geq \tilde{t}:=\left[\exp\left(\frac{\log \hat{t}}{p}\right) \right]^{S_{\max}}>0.
	\end{equation}
	Moreover, in view of \eqref{decrease-property}, the  set
	$$\mathcal{G}:=\left\{ \left[\exp\left(\frac{\log \hat{t}}{p}\right)\right]^{s} \, : \, s\in\{0,1,\ldots,S_{\max}\}   \right\},$$ provides a grid over the interval $(0,1]$. Clearly, by increasing the maximum number of backtracking steps $S_{\max}$, one can include more points in the grid. Moreover, a large value of the parameter $p$ concentrates the grid points toward the right-hand side of the interval $(0,1]$, whereas small values for $\hat{t}\in(0,1)$ cause the grid points to accumulate near the left-hand side of the interval $(0,1]$. The reason for using a different sufficient decrease condition when $s=p$ will become clear in Algorithm~\ref{projected-mifflin-line search} and its convergence analysis.
	
	In case Algorithm~\ref{exp-limited-line-search} terminates with $I_k=1$, we employ the step size $t_k>0$ to take a serious step, i.e., we set 
	\begin{equation}\label{serious_step1}
		\bm x_{k+1}:=P_C(\bm x_k+t_k \bm d_k),
	\end{equation}
	and, for the case $I_k=2$, we set
	\begin{equation}\label{serious_step2}
		\bm x_{k+1}:=P_{C_k}(\bm x_k+t_k \bm d_k).
	\end{equation}
	After taking a serious step, we compute an arbitrary subgradient $\boldsymbol{\xi}_{k+1}\in\partial f(\bm x_{k+1})$, and update the bundle of subgradients by $G_{k+1}:=\{\boldsymbol{\xi}_{k+1}  \}$.
	Eventually, we put $\varepsilon_{k+1}:=\varepsilon_k$, $\delta_{k+1}:=\delta_k$, increment $k$ by one, and repeat the above process.

	\begin{remark}
		In this subsection, we focused on the structure of the proposed algorithm at the $k$-th iteration under the assumption that, for some $\delta_k>0$, the vector $W(\tilde{t}, k)$ satisfies $\lVert W(\tilde{t}, k)\rVert>\delta_k$. Under this assumption, we considered the case in which Algorithm~\ref{exp-limited-line-search} terminates with $I_k\in\{1, 2\}$. The remaining cases, namely when Algorithm~\ref{exp-limited-line-search} terminates with $I_k=0$ or $\lVert W(\tilde{t}, k)\rVert\leq\delta_k$, are addressed in the next subsection.
		
	\end{remark}

	\subsection{ Taking a Null Step and Improving the Bundle of Subgradients}
	
	Suppose that we are at the $k$-th iteration and $\lVert W(\tilde{t}, k)\rVert>\delta_k$. We now consider the case in which Algorithm~\ref{exp-limited-line-search} terminates with $I_k=0$, that is, $t_k=0$.
	In this situation, we take a null step in order to append a new \emph{nonredundant} subgradient to the bundle of subgradients with the aim of improving our approximation of $\partial_{\varepsilon_k} f(\bm x_k)$ over the feasible region $C$. For this purpose, for some $\bm y\in C_k$, we compute $\boldsymbol{\xi}_{k+1}\in \partial f(\bm y)$  that satisfies
	\begin{equation}
		\boldsymbol{\xi}_{k+1}\notin \texttt{conv}\{G_k\}.
	\end{equation}
	Since $\bm g^*_k$ can be expressed by
	$\bm g^*_k=P_{\texttt{conv}\{G_k\}}(\bm 0),$
	it follows from~\eqref{First-Pro-Theorem} that
	\begin{equation*}
		\bm g^T \bm g^*_k\geq \lVert \bm g^*_k\rVert^2, \quad \forall \, \bm g\in \texttt{conv}\{G_k\}.
	\end{equation*}
	By noting that $\bm g^*_k\neq \bm 0$ and $\bm d_k=-\frac{\bm g^*_k}{\lVert\bm g^*_k \rVert}$, the above inequality can be rewritten as 
	\begin{equation}\label{non-redundant-condition}
		\bm g^T \bm d_k\leq -\lVert \bm g^*_k\rVert, \quad \forall \, \bm g\in \texttt{conv}\{G_k\}.
	\end{equation}
	Therefore, by computing $\boldsymbol{\xi}_{k+1}$ which satisfies 
	\begin{equation}\label{suff-cond}
		\boldsymbol{\xi}_{k+1} \bm d_k\geq  -c \lVert \bm g^*_k\rVert,
	\end{equation}
	for some $c\in(0,1)$, condition~\eqref{non-redundant-condition} ensures that $\boldsymbol{\xi}_{k+1}\notin \texttt{conv}\{G_k\}$. In this regard, in Algorithm~\ref{projected-mifflin-line search}, we develop a projected variant of Mifflin's line search \cite{kiwielbook} which employs the sufficient condition~\eqref{non-redundant-condition} to find a nonredundant subgradient. 
	
	\begin{algorithm}
		\caption{A Projected Subgradient Search}
		\label{projected-mifflin-line search}
		\hspace*{\algorithmicindent}\textbf{Inputs:}  Current point $\bm x_k\in C$,  search direction ${\bm d}_k=-\bm g^*_k/\lVert \bm g^*_k\rVert$ such that $\bm g^*_k\neq \bm 0$, along with $\hat{t}$ and $\tilde{t}$ as set in Algorithm~\ref{exp-limited-line-search}. \\
		\hspace*{\algorithmicindent}\textbf{Parameters:}  Sufficient decrease parameter $\beta\in(0,1)$ as chosen in Algorithm~\ref{exp-limited-line-search}, $c\in(0,1)$ with $\beta<c$, and reduction factor $r\in(0,0.5)$. \\
		\hspace*{\algorithmicindent}\textbf{Output:} A nonredundant subgradient $\boldsymbol{\xi}_{k+1}$.\\
		%\hspace*{\algorithmicindent}\textbf{Requirements:} nothing \\
		
		%\hspace*{\algorithmicindent}\textbf{Function:} $\{\bm s, {\rm \bm I}\}$\,\,=\,\,\texttt{T-PLS}\,($\varepsilon$, $\mathbf x$, ${\bm d}$)
		
		\begin{algorithmic}[1]
			\STATE{\textbf{Initialization:} Set $t_0:=\hat{t}\in(0, 1)$, compute arbitrary subgradient $\boldsymbol{\xi}_0\in \partial f\big(P_{C_k}(\bm x_k + t_0\bm d_k)\big)$, and set  $t^l_0=0$, $t^u_0:=1$, and $i:=0$ ;}
			\WHILE{$True$}
			
			\IF{$f\big(P_{C_k}(\bm x_k + t_i \bm d_k) \big)-f\left(\bm x_k \right)\leq - \beta t_i \left\lVert W(\tilde t, k) \right\rVert^2  $ }
			\STATE{ Set $t^l_{i+1}:=t_i$ and $ t^u_{i+1}:=t^u_i$ ;}
			\ELSE
			\STATE{ Set $t^l_{i+1}:=t^l_i$ and $ t^u_{i+1}:=t_i$ ;  }
			\ENDIF
			
			\IF{$\boldsymbol{\xi}_i^T \bm d_k \geq -c  \lVert \bm g^*_k\rVert$}
			\STATE{Set $\boldsymbol{\xi}_{k+1}:=\boldsymbol{\xi}_i$ and \textbf{STOP} ;}
			\ENDIF 
			\STATE{Choose $t_{i+1}\in \left[t^l_{i+1}+r (t^u_{i+1}-t^l_{i+1}), t^u_{i+1}-r (t^u_{i+1}-t^l_{i+1})  \right]$ ;}
			%\IF{$\bm x_k + t_{i+1}\bm d_k\notin C$ and $\eta_i\leq\eta_{\max}$    }
			%\STATE{Compute $\boldsymbol{\xi}_{i+1}\in\partial f(\bm x_k+t_{i+1} \bm d_k)$ and set $\eta_{i+1}:=\eta_i+1$ ;  }
			%\ELSE
			\STATE{ Compute $\boldsymbol{\xi}_{i+1}\in\partial f\big(P_{C_k}(\bm x_k+t_{i+1} \bm d_k)\big)$  ;  }
			%\ENDIF
			\STATE{Set $i:=i+1$ ;}

			\ENDWHILE
			
		\end{algorithmic}
		%\hspace*{\algorithmicindent}\textbf{ End Function}
	\end{algorithm}
	
	Regarding Algorithm~\ref{projected-mifflin-line search}, we provide some explanations. The first conditional block modifies the interval  within which we seek a suitable step size along the search direction $\bm d_k$ to find an effective  subgradient. The second one checks the sufficient condition~\eqref{suff-cond} for the trial subgradient $\boldsymbol{\xi}_i$. As discussed above, once this condition is met, we deduce  $\boldsymbol{\xi}_i\notin\texttt{conv}\{G_k \}$, and therefore, the algorithm is terminated to append $\boldsymbol{\xi}_{k+1}:=\boldsymbol{\xi}_i$ to the bundle of subgradients.
	% By the third conditional block, the algorithm is permitted, for a finite number of iterations, to collect subgradient information at points that are not necessarily contained within the feasible region. The reason behind this strategy is that, due to the use of the projection technique, it is often the case that $\bm x_k$ lies on the boundary of the feasible region, and therefore, to effectively explore $\partial_{\varepsilon_k} f(\bm x_k)$ we may need some subgradient information from the points which do not necessarily belong to the feasible region $C$. Due to some theoretical concerns, we keep the number of such iterations finite.

	In what follows, we show the finite convergence of Algorithm~\ref{projected-mifflin-line search}. To this end, we first need to explore some asymptotic behavior of this algorithm. In this respect, in the following discussion, it is assumed that Algorithm~\ref{projected-mifflin-line search} does not terminate, i.e., $i\to\infty$.
	By construction of this algorithm, one can easily verify that
	\begin{equation}\label{line-p-1}
		0\leq t^l_i\leq t^l_{i+1} < t^u_{i+1}\leq t^u_{i}\leq 1, 
	\end{equation}
	and
	\begin{equation}\label{line-p-2}
		t^u_{i+1}-t^l_{i+1}\leq (1-r)(t^u_i-t^l_i),
	\end{equation}
	for all $i\in\mathbb{N}$.
	% Moreover, by the definition of $C_k$ and $\alpha_k$, it is guaranteed that
	%\begin{equation*}
	%\lVert \bm x_k - P_{C_k}(\bm x_k + t_i \bm d_k)  \rVert\leq \sqrt{n} \lVert \bm x_k - P_{C_k}(\bm x_k + t_i \bm d_k)  \rVert_\infty\leq \sqrt{n}\alpha_k=\varepsilon_k,
	%\end{equation*}
	%for all $i\in\mathbb{N}_0$, which means
	%Moreover, it is guaranteed by \eqref{line-p-1} that $t_i\in(0,\varepsilon_k)$, for all $i\in\mathbb{N}$.
	%As a result, by noting that $\bm x_k\in C$ and $P_C$ is nonexpansive, we have
	%\begin{equation*}
	%	\lVert \bm x_k - P_C(\bm x_k + t_i \bm d_k)  \rVert=\lVert  P_C(\bm x_k) - P_C(\bm x_k + t_i \bm d_k)  \rVert\leq \lVert t_i \bm d_k\rVert = t_i\leq \varepsilon_k,
	%\end{equation*}
	%	for all $i\in\mathbb{N}$. This fact along with $\boldsymbol{\xi}_0\in \partial f\big(P_C(\bm x_k + \varepsilon_k\bm d_k)\big)$
	% implies 
	%\begin{equation}
	%\boldsymbol{\xi}_i\in\partial_{\varepsilon_k} f(\bm x_k), \quad \forall \, i\in\mathbb{N}_0.
	%\end{equation}
	It follows from~\eqref{line-p-1} that the sequences $\{t^u_i\}$ and $\{t^l_i\}$ are convergent. In addition, inequality~\eqref{line-p-2} together with the fact that $r\in(0, 0.5)$ ensures the sequence $\{t^u_i-t^l_i\}$ converges to zero. Consequently, there exists $\bar{t}\in[0, 1]$ such that $t^u_i\downarrow \bar{t}$ and $t^l_i\uparrow \bar{t}$, as $i\to\infty$.  By the first conditional block of the algorithm, we have $t_i\in\{t^l_{i+1}, t^u_{i+1}\}$, for all $i\in\mathbb{N}_0$, which means $t_i\to \bar{t}$, as $i\to\infty$. Next, define
	\begin{equation*}
		\mathrm T_{\text{lower}}:=\left\{t\geq 0 \,\,: \,\, f\big(P_{C_k}(\bm x_k + t \bm d_k)\big)-f(\bm x_k) \leq -\beta t \left\lVert W(\tilde t, k) \right\rVert^2   \right\}.
	\end{equation*} 
	Then, in view of the first conditional block, one can observe $t^l_i\in  \mathrm T_{\text{lower}}$, for all $i\in\mathbb{N}_0$, i.e.,
	\begin{equation*}
		f\left(P_{C_k}(\bm x_k + t^l_i \bm d_k)\right)-f(\bm x_k) \leq -\beta t^l_i \left\lVert W(\tilde{t}, k) \right\rVert^2, \quad \forall \, i\in\mathbb{N}_0.
	\end{equation*}
	Letting $i$ approach infinity in the above inequality, we obtain 
	\begin{equation*}
		f\big(P_{C_k}(\bm x_k + \bar t \bm d_k)\big)-f(\bm x_k) \leq -\beta \bar t \left\lVert   W(\tilde{t}, k)  \right\rVert^2,
	\end{equation*}
	yielding $\bar{t}\in  \mathrm T_{\text{lower}}$.
	
	\begin{lemma}\label{L3}
		Suppose  Algorithm~\ref{exp-limited-line-search} terminates with $I_k=0$, and  Algorithm~\ref{projected-mifflin-line search} does not terminate, i.e., $i\to\infty$ in this algorithm. Let 
		$$\hat{\mathrm{I}}:=\left\{i\in\mathbb{N}_0\,:\, t^u_{i+1}=t_i \,\, \text{\rm in Algorithm~\ref{projected-mifflin-line search}}  \right\}.$$
		Then, $\hat{\mathrm{I}}\subseteq\mathbb{N}_0$ is an infinite set.
	\end{lemma}
	\begin{proof}
		First, we show  $\hat{\mathrm{I}}\subseteq\mathbb{N}_0$ is a nonempty set. Suppose for contradiction that it is an empty set. Then, in virtue of the first conditional block of Algorithm~\ref{projected-mifflin-line search}, one may write
		\begin{equation}\label{L3-1}
			f\big(P_{C_k}(\bm x_k + t_i \bm d_k)\big)-f(\bm x_k)\leq -\beta  t_i\left\lVert W(\tilde{t}, k) \right\rVert^2,
			\quad \forall \, i\in\mathbb{N}_0.
		\end{equation}
		In particular, for $ i=0$, we have $t_0=\hat{t}$ in Algorithm~\ref{projected-mifflin-line search}, and therefore
		\begin{equation}\label{L3-2}
			f\big(P_{C_k}(\bm x_k + \hat{t} \bm d_k)\big)-f(\bm x_k)\leq -\beta \hat t\left\lVert W(\tilde{t}, k) \right\rVert^2.
		\end{equation}
		On the other hand, when $s=p$ in Algorithm~\ref{exp-limited-line-search}, we have 
		$$t=\left[ \exp\frac{\log \hat{t}}{p}    \right]^p=\hat{t}.$$
		Thus, \eqref{L3-2} ensures that Algorithm~\ref{exp-limited-line-search} terminates with $I_k=2$, which is a contradiction.  Next, we prove that $\hat{\mathrm{I}}$ is infinite. By indirect proof, suppose that $\hat{\mathrm{I}}$ is a finite set. Then, as $\hat{\mathrm{I}}\neq\emptyset$ and $t^u_i\downarrow \bar t$, as $i\to\infty$, there exists $\bar i\in\mathbb{N}$ such that  
		\begin{equation*}
			t^u_i=\bar t, \quad \forall \, i\geq \bar{i} \quad \text{and} \quad t^u_i>\bar{t}, \quad \forall \, i<\bar{i}. 
		\end{equation*}
		Consequently, we have $t^u_{\bar i}=t_{\bar i -1}=\bar{t}$, and hence 
		\begin{equation*}
			f\big(P_{C_k}(\bm x_k + \bar{t} \bm d_k)\big)-f(\bm x_k) > -\beta \bar t\left\lVert W(\tilde{t}, k) \right\rVert^2,
		\end{equation*} 
		violating the fact that $\bar{t}\in\mathrm{T}_{\text{lower}}$.
	\end{proof}

	Now, we are ready to show the finite convergence of Algorithm~\ref{projected-mifflin-line search}.
	
	\begin{theorem}
		Assume that $f:\mathbb{R}^n\to\mathbb{R}$ is a weakly semismooth function. If Algorithm~\ref{exp-limited-line-search} terminates with $I_k=0$, then Algorithm~\ref{projected-mifflin-line search} terminates after a finite number of iterations.
	\end{theorem}
	\begin{proof}
		Suppose for contradiction that Algorithm~\ref{projected-mifflin-line search} does not terminate. Then, by Lemma~\ref{L3}, $\hat{\mathrm{I}}\subseteq\mathbb{N}_0$ is an infinite set. Moreover,
		\begin{equation}\label{T3-1}
			f\big(P_{C_k}(\bm x_k + t_i \bm d_k)\big)-f(\bm x_k) > -\beta  t_i\left\lVert W(\tilde{t}, k)\right\rVert^2, \quad \forall\, i\in\hat{\mathrm{I}}.
		\end{equation}
		Recalling $\bar{t}\in\mathrm{T}_{\text{lower}}$, we have 
		\begin{equation}\label{T3-2}
			f\big(P_{C_k}(\bm x_k + \bar{t} \bm d_k)\big)-f(\bm x_k) \leq -\beta \bar t\left\lVert W(\tilde{t}, k)\right\rVert^2.
		\end{equation}
		%	Notice that, for each $i\in\hat{\mathrm{I}}$, we have $t_i=t^u_{i+1}>\bar t>0$. Thus, in view of \eqref{projection-feature2}, one may see
		%	\begin{equation}
		%	t_i\left\lVert \frac{1}{ t_i} [ \bm x_k - P_C(\bm x_k-  t_i\bm g^*_k)]\right\rVert \leq \bar t\left\lVert \frac{1}{\bar t} [ \bm x_k - P_C(\bm x_k- \bar t\bm g^*_k)]\right\rVert, \quad \forall \, i\in\mathrm{I}.
		%	\end{equation}
		%	Hence, for any $i\in\mathrm{I}$, we can continue \eqref{T3-2} with 
		%	\begin{align}\label{T3-3}
		%	f\big(P_C(\bm x_k + \bar{t} \bm d_k)\big)-f(\bm x_k)& \leq -\beta \bar t\left\lVert \frac{1}{\bar t} [ \bm x_k - P_C(\bm x_k- \bar t\bm g^*_k)]\right\rVert^2\nonumber\\& \leq-\beta  t_i \left\lVert \frac{1}{ t_i} [ \bm x_k - P_C(\bm x_k-  t_i\bm g^*_k)]\right\rVert^2\nonumber\\& \leq-\beta  \bar t \left\lVert \frac{1}{ t_i} [ \bm x_k - P_C(\bm x_k-  t_i\bm g^*_k)]\right\rVert^2.
		%	\end{align}
		By employing  \eqref{T3-1} and \eqref{T3-2}, for each $i\in\hat{\mathrm{I}}$, we arrive at the following inequality
		\begin{align*}
			f\big(P_{C_k}(\bm x_k + t_i \bm d_k)\big) - 	f\big(P_{C_k}(\bm x_k + \bar{t} \bm d_k)\big)> -\beta \left(t_i-\bar t \,\right) \left\lVert W(\tilde{t}, k)\right\rVert^2.
		\end{align*}
		Since $\bm x_k\in C$ and $P_{C}$ is a nonexpansive map,  one can write
		\begin{align*}
			\left\lVert W(\tilde{t}, k)\right\rVert^2&=\frac{1}{{\tilde t}^2} \left\lVert \bm x_k-P_C(\bm x_k- \tilde t\bm g^*_k) \right\rVert^2=\frac{1}{{\tilde t}^2} \left\lVert P_C(\bm x_k)-P_C(\bm x_k- \tilde t\bm g^*_k) \right\rVert^2\\& \leq \frac{1}{\tilde t^2} \lVert \tilde t \bm g^*_k \rVert^2 = \lVert  \bm g^*_k \rVert^2,
		\end{align*}
		and therefore
		\begin{equation*}
			f\big(P_{C_k}(\bm x_k + t_i \bm d_k)\big) - 	f\big(P_{C_k}(\bm x_k + \bar{t} \bm d_k)\big)> -\beta \left(t_i-\bar t\,\right) \rVert   \bm g^*_k\rVert^2,
		\end{equation*}
		for all $i\in\hat{\mathrm{I}}$. Let $h_i:=t_i-\bar{t}>0$, for all $i\in\hat{\mathrm{I}}$, and $\bm z_k:=\bm x_k+ \bar{t} \bm d_k$. Then, the above inequality can be represented as
		\begin{equation}\label{T3-4}
			-\beta  \rVert   \bm g^*_k\rVert^2 <	\frac{f\big(P_{C_k}(\bm z_k + h_i \bm d_k)\big) - 	f\big(P_{C_k}(\bm z_k)\big)}{h_i},\quad \forall \, i\in\hat{\mathrm{I}}.
		\end{equation}
		Moreover, $h_i\downarrow 0$ as $i\xrightarrow{\hat{\mathrm{I}}}\infty$, and
		%	As $\eta_{\max}\in\mathbb{N}$ and the index set $\mathrm{I}$ is an infinite set, the third conditional block of Algorithm~\ref{projected-mifflin-line search} ensures the existence of $\bar{i}\in\mathbb{N}$ such that
		\begin{equation}\label{T3-5}
			\boldsymbol{\xi}_i\in\partial f\big(P_{C_k}(\bm x_k+t_i\bm d_k)     \big)=\partial f\big(P_{C_k}(\bm z_k+h_i\bm d_k)     \big), \quad \forall\, i\in\hat{\mathrm{I}}.
		\end{equation}
		Since $f:\mathbb{R}^n\to\mathbb{R}$ is weakly semismooth, by Theorem~\ref{T0},  we know that $f\circ P_{C_k}$ is weakly semismooth, as well. This fact along with \eqref{T3-4} and \eqref{T3-5}  implies
		\begin{align}\label{T3-6}
			-\beta  \rVert   \bm g^*_k\rVert^2& \leq \lim_{i\xrightarrow{\hat{\mathrm{I}}}\infty}	\frac{f\big(P_{C_k}(\bm z_k + h_i \bm d_k)\big) - 	f\big(P_{C_k}(\bm z_k)\big)}{h_i} = \left(f\circ P_{C_k}\right)'(\bm z_k; \bm d_k)\nonumber\\&=\lim_{i\xrightarrow{\hat{\mathrm{I}}}\infty} \boldsymbol{\xi}_i^T \bm d_k.
		\end{align}
		On the other hand, as  Algorithm~\ref{projected-mifflin-line search} does not terminate, it must be the case that
		\begin{equation*}
			\boldsymbol{\xi}_i^T \bm d_k < -c \lVert\bm g^*_k \rVert,\quad \forall\, i\in\mathbb{N}_0,
		\end{equation*}
		which means
		\begin{equation*}
			\lim_{i\xrightarrow{\hat{\mathrm{I}}}\infty} \boldsymbol{\xi}_i^T \bm d_k \leq -c \lVert\bm g^*_k \rVert < -\beta \lVert\bm g^*_k \rVert,
		\end{equation*}
		contradicting \eqref{T3-6}.
	\end{proof}
	
	Eventually, we show that, for the subgradient $\boldsymbol{\xi}_{k+1}$ generated by Algorithm~\ref{projected-mifflin-line search}, we have  $\boldsymbol{\xi}_{k+1}\in\partial_{\varepsilon_k} f(\bm x_k)$.
	
	\begin{lemma}\label{L-b-last}
		The resulting subgradient $\boldsymbol{\xi}_{k+1}$ generated by Algorithm~\ref{projected-mifflin-line search} satisfies $$\boldsymbol{\xi}_{k+1}\in\partial_{\varepsilon_k} f(\bm x_k).$$
	\end{lemma}
	\begin{proof}
		At the $i$-th iteration of Algorithm~\ref{projected-mifflin-line search},  by the definition of $C_k$ and $\alpha_k$, it is guaranteed that
		\begin{equation*}
			\lVert \bm x_k - P_{C_k}(\bm x_k + t_{i+1} \bm d_k)  \rVert\leq \sqrt{n} \lVert \bm x_k - P_{C_k}(\bm x_k + t_{i+1} \bm d_k)  \rVert_\infty\leq \sqrt{n}\alpha_k=\varepsilon_k.
		\end{equation*}
		Thus, if the algorithm terminates at the $\bar{i}$-th iteration, it must be the case that $\boldsymbol{\xi}_{k+1}:=\boldsymbol{\xi}_{\bar i}\in\partial_{\varepsilon_k} f(\bm x_k)$.
	\end{proof}
	
	Once Algorithm~\ref{projected-mifflin-line search} computes the nonredundant subgradient  $\boldsymbol{\xi}_{k+1}$, 	we set 
	\begin{equation*}
		\bm x_{k+1}:=\bm x_k \quad \text{and} \quad G_{k+1}:=G_k\cup \{\boldsymbol{\xi}_{k+1}  \}.
	\end{equation*} 
	In addition, we put $\varepsilon_{k+1}:=\varepsilon_k$ and $\delta_{k+1}:=\delta_k$, increment $k$ by one, and the process of Subsection~\ref{subsec1} is repeated.
	
	Termination of Algorithm~\ref{exp-limited-line-search} with $I_k=0$ is not the only situation in which we take a null step. Assume that the norm of $W(\tilde t, k)$ is not sufficiently large, i.e.,
	$$\left\lVert W(\tilde t, k)\right\rVert\leq \delta_k.$$
	In this case, motivated by the following lemma, we take a null step by setting $\bm x_{k+1}:=\bm x_k$ in order to update $\varepsilon_k$ and $\delta_k$ by 
	$$\varepsilon_{k+1}:=\sigma\varepsilon_k\quad \text{and} \quad \delta_{k+1}:=\sigma\delta_k,$$
	in which $\sigma\in(0,1)$ is a reduction factor. Next, we update $G_k$ by 
	$$G_{k+1}:=\{ \boldsymbol{\xi}_{k+1}  \},$$
	where $\boldsymbol{\xi}_{k+1}\in\partial f(\bm x_{k+1})$ is an arbitrary subgradient.
	%\{ \boldsymbol{\xi}\in G_k \,\, :\,\, \boldsymbol{\xi}\in\partial_{\varepsilon_{k+1}} f(\bm x_{k}) \},$$
	Finally, we increment $k$ by one, and repeat the process of Subsection~\ref{subsec1}.
	
	\begin{lemma}\label{L4'}
		Assume that there exists an infinite subset $\mathcal{K}\subseteq\mathbb{N}_0$ such that, for some $\bm x^*\in C$, the sequence $\{\bm x_k  \}_{k\in\mathcal{K}}$ converges to $\bm x^*$.  Suppose, in addition, 
		$$\max\left\{ \left\lVert W(\tilde t, k)\right\rVert, \varepsilon_k   \right\}\to 0, \quad \text{\rm as} \quad k\xrightarrow{\mathcal{K}}\infty.$$
		Then, $\bm x^*$ is stationary for problem~\eqref{Main_Problem-1}.
	\end{lemma}
	\begin{proof}
		By assumption, we have $ \bm x_k-P_C(\bm x_k-\tilde{t}\bm g^*_k)\to \bm 0$ and $\varepsilon_k\to 0$, as $k\xrightarrow{\mathcal{K}}\infty$. Moreover,
		$\bm g^*_k\in\texttt{conv}\{G_k  \}\subseteq\partial_{\varepsilon_k} f(\bm x_k)$,  for all  $k\in\mathcal{K}.$
		As $\bm x_k\xrightarrow{k\in\mathcal{K}}\bm x^*$,  the local boundedness of the Clarke $\varepsilon$-subdifferential map together with its upper semicontinuity  implies the existence of $\mathcal{K}'\subseteq\mathcal{K}$ and $\boldsymbol{\xi}^*\in\partial f(\bm x^*)$ such that $\bm g^*_k\to\boldsymbol{\xi}^*$, as $k\xrightarrow{k\in\mathcal{K}'}\infty$.  Thus,  
		$$\bm x^*=P_C(\bm x^* - \tilde t \boldsymbol{\xi}^*).$$
		Since $\tilde{t}>0$ and $\boldsymbol{\xi}^*\in\partial f(\bm x^*)$, it follows from Theorem~\ref{T2}  that $\bm x^*$ is stationary for problem~\eqref{Main_Problem-1}.
	\end{proof}
	%Inspired by the above lemma, once $\lVert\frac{1}{\tilde{t}}[\bm x_k-P_C(\bm x_k-\tilde{t}\bm g^*_k)]\rVert\leq \delta_k$, we measure the stationarity of the current point $\bm x_{k+1}$ by
	%\begin{equation}
	%v_{k+1}:=\max\left\{ \left\lVert\frac{1}{\tilde{t}}[\bm x_k-P_C(\bm x_k-\tilde{t}\bm g^*_k)]\right\rVert, \varepsilon_k   \right\}.
	%\end{equation}

	\subsection{ Main Algorithm}
	
	Based on the provided details in the previous parts, the proposed method is presented in Algorithm~\ref{main-alg}. In this algorithm, inspired by Lemma~\ref{L4'}, the variable $v_k$ serves as an optimality certificate for the method's overall performance up to the $k$-th iteration. It is initialized with $v_0:=\infty$ and is updated according to
	\begin{equation*}
		v_{k+1}:=\max\left\{ \left\lVert W(\tilde t, k)\right\rVert, \varepsilon_k   \right\},
	\end{equation*}
	once a significant reduction in  $ \left\lVert W(\tilde t, k)\right\rVert$ is observed.

	\begin{algorithm}
		\caption{Main Algorithm}
		\label{main-alg}
		\hspace*{\algorithmicindent}\textbf{Inputs:} Starting point $\bm x_0\in C$, $(\varepsilon_0, \delta_0)\in(0,1)\times(0, 1)$, and stationarity tolerance $\tau>0$.\\
		\hspace*{\algorithmicindent}\textbf{Parameters:}   Reduction factor $\sigma\in(0,1)$, and $\tilde t$ as set in Algorithm~\ref{exp-limited-line-search}. \\
		\hspace*{\algorithmicindent}\textbf{Output:} An approximate stationary point. 
		%\hspace*{\algorithmicindent}\textbf{Requirements:} nothing \\
		
		%\hspace*{\algorithmicindent}\textbf{Function:} $\{\bm s, {\rm \bm I}\}$\,\,=\,\,\texttt{T-PLS}\,($\varepsilon$, $\mathbf x$, ${\bm d}$)
		
		\begin{algorithmic}[1]
			\STATE{\textbf{Initialization:}  Compute arbitrary subgradient $\boldsymbol{\xi}_0\in \partial f(\bm x_0)$, and set $G_0:=\{ \boldsymbol{\xi}_0  \}$, $v_0:=+\infty$, and $k:=0$ ;}
			\WHILE{$v_k> \tau$}
			
			\STATE{ Compute $\bm g^*_k$ by solving subproblem \eqref{least-norm} ;  }
			
			\IF{$\lVert W(\tilde t, k) \rVert\leq \delta_k$  }
			\STATE{Set $v_{k+1}:=\max\{ \lVert W(\tilde t, k) \rVert, \varepsilon_k  \}$ ;   }
			\STATE{Set $\varepsilon_{k+1}:=\sigma\varepsilon_k, \delta_{k+1}:=\sigma\delta_k,$ and $\bm x_{k+1}:=\bm x_k$ ;}
			\STATE{For an arbitrary subgradient $\boldsymbol{\xi}_{k+1}\in\partial f(\bm x_{k+1})$, set $G_{k+1}:=\{  \boldsymbol{\xi}_{k+1} \}$ ; }
			\STATE{Set $k\leftarrow k+1$, and \textbf{Continue} ;}
			\ENDIF
			
			\STATE{Compute the normalized search direction  $\bm d_k$ by $\bm d_k:=-\bm g^*_k/\lVert\bm g^*_k\rVert$ ;}
			\STATE{Compute $t_k$ and $I_k$ by using  Algorithm~\ref{exp-limited-line-search} ;   }

			\IF{$I_k=1$}
			
			\STATE{Set $\bm x_{k+1}:=P_C(\bm x_k+ t_k \bm d_k)$, and compute $\boldsymbol{\xi}_{k+1}\in\partial f(\bm x_{k+1})$ ;  }
			\STATE{Set $G_{k+1}:=\{ \boldsymbol{\xi}_{k+1} \}$ ;}
			\STATE{Set $\varepsilon_{k+1}:=\varepsilon_k$ and $\delta_{k+1}:=\delta_k$ ;}
			%\STATE{Set $v_{k+1}:=\max\{\lVert \bm x_{k+1}-\bm x_k\rVert, \varepsilon_k   \}$ ;}
			
			\ELSIF{$I_k=2$}
			
			\STATE{Set $\bm x_{k+1}:=P_{C_k}(\bm x_k+ t_k \bm d_k)$, and compute $\boldsymbol{\xi}_{k+1}\in\partial f(\bm x_{k+1})$ ;  }
			\STATE{Set $G_{k+1}:=\{ \boldsymbol{\xi}_{k+1} \}$ ;}
			\STATE{Set $\varepsilon_{k+1}:=\varepsilon_k$ and $\delta_{k+1}:=\delta_k$ ;}
			%		      \IF{ $\lVert \bm x_{k+1} -\bm x_k\rVert\leq \delta_k$ }
			%		      \STATE{Set $\varepsilon_{k+1}:=\sigma\varepsilon_k$ and $\delta_{k+1}:=\sigma\delta_k$ ; }
			%		      \ENDIF
			\ELSE
			\STATE{ Compute $\boldsymbol{\xi}_{k+1}\in\partial_{\varepsilon_k}f(\bm x_k)$ by using Algorithm~\ref{projected-mifflin-line search} ;} \STATE{ Set $G_{k+1}:=G_k\cup \{\boldsymbol{\xi}_{k+1}\}$ ; }
			\STATE{Set $\varepsilon_{k+1}:=\varepsilon_k, \delta_{k+1}:=\delta_k,$ and $\bm x_{k+1}:=\bm x_k$ ; }

			\ENDIF
			
			\STATE{ Set $v_{k+1}:=v_k$ ;}
			\STATE{Set $k\leftarrow k+1$ ;}

			\ENDWHILE
			
		\end{algorithmic}
		%\hspace*{\algorithmicindent}\textbf{ End Function}
	\end{algorithm}
	
	\section{Convergence Analysis}\label{Conv-Analysis}
	In order to observe the asymptotic behavior of Algorithm~\ref{main-alg}, throughout this section, it is assumed that $\tau=0$, which allows the algorithm to generate the infinite sequence $\{\bm x_k\}_{k\in\mathbb{N}_0}$. Next, one of the following cases may occur:\\
	\textbf{Case (i):} The number of serious steps is infinite.\\
	\textbf{Case (ii):} The number of serious steps is finite.

	First, we consider case (i), where the index set 
	\begin{equation*}
		K:=\{k\in\mathbb{N}_0 \,\, : \,\, I_k=1\, \text{or} \, I_k=2 \,\, \text{in Algorithm~\ref{main-alg}} \},
	\end{equation*}
	is infinite. Throughout the study of this case, we need the following assumption.
	\begin{assumption}\label{assump1}
		The sublevel set 
		$$lev_{f(\bm x_0)}( f):=\{\bm x\in C \,\,:\,\, f(\bm x)\leq f(\bm x_0)   \},$$
		is a bounded set.
	\end{assumption}
	
	We first need the following technical lemma.
	\begin{lemma}\label{L4}
		Suppose Assumption~\ref{assump1} holds, and the index set $K$ is infinite. Then, $\delta_k\downarrow 0$ and $\varepsilon_k\downarrow 0 $,  as $k\to\infty$. 
	\end{lemma}
	\begin{proof}
		Since $\delta_k$ and $\varepsilon_k$ are simultaneously updated in Algorithm~\ref{main-alg}, it is sufficient to show that $\delta_k\downarrow 0$ as $k\to\infty$. By indirect proof, assume the existence of $\bar{k}\in\mathbb{N}_0$ such that 
		$\delta_k=\delta_{\bar{k}}>0$, for all $k\geq \bar{k}$. Thus, construction of Algorithm~\ref{main-alg} implies
		\begin{equation*}
			\lVert W(\tilde t, k) \rVert > \delta_{\bar{k}}, \quad \forall \, k \geq \bar{k}.
		\end{equation*} 
		For any $k\in K$, as a serious step takes place, we have 
		\begin{align}\label{L4-1}
			f(\bm x_{k+1})-f(\bm x_k)&\leq -\beta t_k \left\lVert W(\tilde t, k) \right\rVert^2 \leq -\beta\, \tilde t\, \delta_{\bar k}^2.
		\end{align}
		Moreover, for each $k\in\mathbb{N}_0\setminus K$, we have $f(\bm x_{k+1})=f(\bm x_k)$. This fact along with $\eqref{L4-1}$ yields
		\begin{equation*}
			f(\bm x_{k+1})\leq f(\bm x_0)-\sum_{\substack{j\in K\\ j\leq k+1}} \beta\, \tilde{t}\, \delta_{\bar k}^2.
		\end{equation*} 	
		As $K$ is infinite, letting $k$ approach infinity in the above inequality, we obtain $f(\bm x_k)\to -\infty$, as $k\to\infty$. On the other hand, by Assumption~\ref{assump1}, we have 
		\begin{equation*}
			\inf \left\{f(\bm x) \,\,: \,\, \bm x \in C   \right\} \in \mathbb{R},
		\end{equation*}
		which is a contradiction.
	\end{proof}
	The main results for case (i) is stated in the next theorem.
	
	\begin{theorem}\label{T4}
		Suppose Assumption~\ref{assump1} holds, and the index set $K$ is infinite.
		\begin{itemize}
			\item[(i)] Let $K^*:=\{k\in\mathbb{N}_0\,:\, \lVert W(\tilde{t},k) \rVert\leq\delta_k  \}$. Then, any limit point of the sequence $\{\bm x_k \}_{k\in K^*}$ is stationary for problem \eqref{Main_Problem-1}.
			\item[(ii)] Any limit point of the sequence  $\{\bm x_k\}_{k\in K}$ is stationary for problem \eqref{Main_Problem-1}.
			\item[(iii)] 
			$f(\bm x_k)\downarrow f(\bm x^*) $ as $k\to\infty$, where $\bm x^*\in C$ is a stationary point for problem~\eqref{Main_Problem-1}.
		\end{itemize}
		
		%any limit point of the sequence $\{\bm x_k  \}_{k\in\mathbb{N}_0}$ generated by Algorithm~\ref{main-alg} is stationary for problem~\eqref{Main_Problem-1}.
	\end{theorem}
	
	\begin{proof}
		(i) It follows from  Lemma~\ref{L4} that $\delta_k\downarrow 0$ and $\varepsilon_k\downarrow 0$,   as $k\to\infty$. Consequently,  construction of Algorithm~\ref{main-alg} implies that $K^*$ is an infinite set. By the descent nature of Algorithm~\ref{main-alg}, we have $\{\bm x_k \}_{k\in K^*}\subseteq lev_{f(\bm x_0)}(f)$, and thus, Assumption~\ref{assump1} ensures that the sequence 	$\{\bm x_k \}_{k\in K^*}$ is bounded.  Let $\mathcal{K}\subseteq K^*$ be such that $\bm x_k\xrightarrow{k\in\mathcal{K}}\bm x^*$. As $C$ is closed and $\{\bm x_k\}_{k\in\mathcal{K}}\subseteq C$, we deduce $\bm x^*\in C$. Moreover, for any $k\in\mathcal{K}\subseteq K^*$, we have
		\begin{equation*}
			\lVert W(\tilde{t}, k) \rVert\leq \delta_k,
		\end{equation*}
		which means
		\begin{equation*}
			\max \{ \lVert W(\tilde{t}, k) \rVert, \varepsilon_k  \} \to 0 \quad \text{as} \quad k\xrightarrow{k\in\mathcal{K}}\infty.
		\end{equation*}
		Now, it follows from Lemma~\ref{L4'} that $\bm x^*\in C$ is a stationary point for problem~\eqref{Main_Problem-1}. 
		
		(ii) Lemma~\ref{L4} implies  $\delta_k\downarrow 0$ and $\varepsilon_k\downarrow 0$,   as $k\to\infty$. For any $k\in K$, since a serious step occurs, we have
		\begin{equation}\label{Th1-1}
			0\leq \beta t_k \lVert W(\tilde{t}, k) \rVert^2 \leq f(\bm x_k)-f(\bm x_{k+1}).
		\end{equation}
		The descent nature of Algorithm~\ref{main-alg} ensures $\{\bm x_k \}_{k\in K}\subseteq lev_{f(\bm x_0)}(f)$. Hence, let $\mathcal{K}\subseteq K$ be such that $\bm x_k\xrightarrow{k\in\mathcal{K}}\bm x^*$. Clearly, $\bm x^*\in C$. Continuity of $f$ yields $f(\bm x_k)\xrightarrow{k\in\mathcal{K}} f(\bm x^*)$. This fact along with the monotonicity of the sequence $\{f(\bm x_k) \}_{k\in\mathbb{N}_0}$ gives 
		\begin{equation*}
			f(\bm x_k)\downarrow f(\bm x^*) \quad \text{as} \quad k\to\infty.
		\end{equation*}
		Therefore, by taking $t_k\geq \tilde{t}>0$ for all $k\in K$, and $\beta>0$ into account, we conclude from \eqref{Th1-1} that
		\begin{equation*}
			\lVert W(\tilde{t}, k) \rVert\to 0 \quad \text{as} \quad k\xrightarrow{k\in\mathcal{K}}\infty,
		\end{equation*}
		yielding
		\begin{equation*}
			\max \{ \lVert W(\tilde{t}, k) \rVert, \varepsilon_k  \} \to 0 \quad \text{as} \quad k\xrightarrow{k\in\mathcal{K}}\infty.
		\end{equation*}
		Consequently, we deduce from Lemma~\ref{L4'} that  $\bm x^*\in C$ is a stationary point for problem~\eqref{Main_Problem-1}. 
		
		(iii) This follows immediately from the proof of part(ii) of this theorem.
		% Let $\mathcal{K}\subseteq K^*$ be as defined in the proof of part(i) of this theorem. Then,  $\bm x_k\xrightarrow{k\in\mathcal{K}}\bm x^*$ with  $\bm x^*\in C$ as a stationary point for problem~\eqref{Main_Problem-1}. Continuity of $f$ implies $f(\bm x_k)\xrightarrow{k\in\mathcal{K}} f(\bm x^*)$, and it follows from the monotonicity of the sequence $\{f(\bm x_k) \}_{k\in\mathbb{N}_0}$ that 
		%\begin{equation*}
		%f(\bm x_k)\downarrow f(\bm x^*) \quad \text{as} \quad k\to\infty.
		%\end{equation*}
	\end{proof}
	
	\begin{remark}
		Theorem~\ref{T4} identifies two subsequences of the sequence of iterates $\{\bm x_k\}_{k\in\mathbb{N}_0}$ generated by Algorithm~\ref{main-alg}, each of which has the property that any cluster point is stationary for the main problem. These results seem to be the strongest stationarity guarantees that can reasonably be obtained under the present assumptions. In the unconstrained setting, several conditions and frameworks have been proposed in~\cite{Improved-convergence} to ensure that every cluster point of the generated sequence of iterates is stationary. Despite our efforts, extending such conditions and frameworks to the present projection-based setting appears to be difficult, leaving us with the open question of whether every cluster point of the sequence of iterates is stationary.
	\end{remark}

	Next, we consider case (ii), where the number of serious steps is finite.
	
	\begin{lemma}\label{T5}
		Suppose that the index set $K$ is finite. Then
		\begin{equation*}
			\liminf_{k\to\infty} \lVert W(\tilde{t}, k) \rVert=0.
		\end{equation*}
	\end{lemma}
	\begin{proof}
		Since $K$ is finite, there exits $\tilde{k}\in\mathbb{N}_0$ such that  
		\begin{equation}
			\bm x_k = \bm x_{\tilde k}, \quad \forall \, k\geq \tilde{k}.
		\end{equation}
		The local boundedness of the Clarke $\varepsilon$-subdifferential map along with the fact that $\bm g^*_k\in\texttt{conv} \{G_k\}\subseteq\partial_{\varepsilon_k} f(\bm x_k)$, for all $k\in\mathbb{N}_0$, ensures the sequence $\{\lVert W(\tilde{t},k)\rVert  \}_{k\in\mathbb{N}_0}$ is bounded. Let
		\begin{equation}
			w:=\liminf_{k\to\infty} \lVert W(\tilde{t}, k) \rVert.
		\end{equation}
		For contradiction purposes, assume $w>0$. Then, by construction of Algorithm~\ref{main-alg}, there exists $k'\in\mathbb{N}_0$ such that 
		\begin{equation}
			G_{k+1}=G_k\cup \{\boldsymbol{\xi}_{k+1} \}, \quad \forall \, k\geq k'.
		\end{equation}
		Furthermore, $w>0$ implies $\liminf_{k\to\infty} \lVert \bm g^*_k \rVert >0$. In other words,  there exist $\bar{w}>0$ and $k''\in\mathbb{N}_0$ such that
		\begin{equation}\label{T5-liminf-cons}
			\lVert \bm g^*_k \rVert \geq \bar{w}, \quad \forall \, k\geq k''.
		\end{equation}
		Let $\bar{k}:=\max\{\tilde{k}, k', k''  \}$. By construction of Algorithm~\ref{projected-mifflin-line search}, we have $$\boldsymbol{\xi}_{k+1}^T \bm d_k\geq -c \lVert \bm g^*_k \rVert,\quad \forall k\geq \bar{k}.$$ 
		Taking $\bm d_k:=-\bm g^*_k/\lVert\bm g^*_k \rVert$ into account, the above inequality  is equivalent to
		\begin{equation}\label{T5-v-inequality}
			\boldsymbol{\xi}_{k+1}^T \bm g^*_k\leq  c \lVert \bm g^*_k \rVert^2, \quad \forall k\geq \bar{k}.
		\end{equation}
		Moreover, by Lemma~\ref{L-b-last},
		\begin{equation*}
			\boldsymbol{\xi}_{k+1}\in\partial_{\varepsilon_k}f(\bm x_k)= \partial_{\varepsilon_{\bar k}}f(\bm x_{\bar k}), \quad \forall \, k\geq \bar{k}.
		\end{equation*}
		Notice that the above equality follows from the fact that $\bm x_k=\bm x_{\bar k}$ and $\varepsilon_k=\varepsilon_{\bar k}$, for all $k\geq \bar k$. Define $C_1:=\max\{\lVert \boldsymbol{\xi} \rVert \,:\, \boldsymbol{\xi}\in\partial_{\varepsilon_{\bar k}} f(\bm x_{\bar{k}})   \}$, and $C_2:=\max\{C_1, \bar{w} \}$. Then, in view of $\bm g^*_k, \boldsymbol{\xi}_{k+1}\in \texttt{conv}\{G_{k+1}\}\subseteq \partial_{\varepsilon_{\bar k}}f(\bm x_{\bar k})$, for all $k\geq \bar{k}$, we conclude
		\begin{equation}\label{T5-up-bound}
			\lVert \boldsymbol{\xi}_{k+1}- \bm g^*_k \rVert \leq 2 C_1 \leq 2 C_2, \quad \forall \, k\geq \bar{k}.
		\end{equation}
		Next, for any $\lambda\in(0,1)$ and $k\geq \bar{k}$, we have
		\begin{align}\label{T5-cont.}
			\lVert \bm g^*_{k+1} \rVert^2 &\leq \lVert \lambda \boldsymbol{\xi}_{k+1} + (1-\lambda)\bm g^*_k \rVert^2\nonumber \\& 
			= \lambda^2 \lVert \boldsymbol{\xi}_{k+1} - \bm g^*_k \rVert^2 + 2\lambda {\bm g^*_k}^T (\boldsymbol{\xi}_{k+1} - \bm g^*_k)+ \lVert \bm g^*_k\rVert^2.
		\end{align}
		In virtue of \eqref{T5-v-inequality} and \eqref{T5-up-bound}, one can continue \eqref{T5-cont.} as
		\begin{align}\label{T5-g_star}
			\lVert \bm g^*_{k+1} \rVert^2&\leq 4\lambda^2 C_2^2 +2\lambda c   \lVert \bm g^*_k\rVert^2 -2\lambda   \lVert \bm g^*_k\rVert^2  +\lVert \bm g^*_k\rVert^2\nonumber\\&
			=4\lambda^2 C_2^2 +\big( 1-2\lambda (1-c)  \big) \lVert \bm g^*_k \rVert^2=: q(\lambda),
		\end{align}
		for all $\lambda\in(0, 1)$.
		One may verify that $\lambda^*:=(1-c)\lVert \bm g^*_k \rVert^2 /4C_2^2 \in (0,1)$ minimizes $q(\lambda)$, and
		\begin{equation*}
			q(\lambda^*)= \left(1-\frac{(1-c)^2\lVert \bm g^*_k \rVert^2}{4C_2^2} \right)\lVert \bm g^*_k \rVert^2.
		\end{equation*}
		By taking \eqref{T5-liminf-cons} into account, we can observe
		\begin{equation}\label{T5-q_t}
			q(\lambda^*)\leq \left(1-\frac{(1-c)^2\bar{w}^2}{4C_2^2} \right)\lVert \bm g^*_k \rVert^2.
		\end{equation}
		Noting that  $\bar{w}\leq C_2$ and $c\in(0, 1)$, one can deduce $\sigma:=1-\frac{(1-c)^2\bar{w}^2}{4C_2^2}\in(0, 1)$. Now, \eqref{T5-g_star} and \eqref{T5-q_t} yield
		\begin{equation}\label{T5-last}
			0\leq \lVert \bm g^*_{k+1} \rVert^2\leq q(\lambda^*)\leq \sigma \lVert \bm g^*_k \rVert^2, \quad \forall \, k\geq \bar{k},
		\end{equation}
		which means the sequence $\{\lVert \bm g^*_k \rVert \}_{k\in\mathbb{N}_0}$ is convergent. Let $\lVert \bm g^*_k \rVert\to l$, as $k\to \infty$. Since $\sigma\in(0, 1)$, we conclude from \eqref{T5-last} that $l=0$, and therefore 
		$$\bm g^*_k\to \bm 0, \quad \text{as} \quad k\to \infty,$$
		contradicting  \eqref{T5-liminf-cons}.
	\end{proof}
	
	Concerning case (ii), the main result is provided in the following theorem.
	
	\begin{theorem}
		Suppose that the index set $K$ is finite, and
		$$\bar k:=\max\{k\,\,: \,\, k\in K  \},$$
		with the convention that  $\bar k:=0$ if $K=\emptyset$. 	Then, $\bm x^*:=\bm x_{\bar k + 1}$ is a stationary point for problem~\eqref{Main_Problem-1}.
	\end{theorem}
	
	\begin{proof}
		By assumption, $\bm x_k=\bm x_{\bar k+1}$, for all $k>\bar k$, and therefore
		\begin{equation}
			\bm x_k\to \bm x^* \quad \text{as} \quad k\to\infty.
		\end{equation}
		Clearly, $\bm x^*\in C$. Moreover, since $K$ is finite, it follows from Lemma~\ref{T5} that 
		\begin{equation}\label{T6-1}
			\liminf_{k\to\infty} \lVert W(\tilde{t}, k) \rVert \to 0.
		\end{equation}
		Now, the construction of Algorithm~\ref{main-alg} implies $\varepsilon_k\downarrow 0$, as $k\to\infty$. This fact together with \eqref{T6-1} ensures the existence of $\mathcal{K}\subseteq\mathbb{N}_0$ such that
		$$\max\{ \lVert W(\tilde{t}, k)\rVert, \varepsilon_k  \}\to 0 \quad \text{as} \quad k\xrightarrow{k\in\mathcal{K}}\infty.$$
		Next, it follows from Lemma~\ref{L4'} that $\bm x^*$ is stationary for problem~\eqref{Main_Problem-1}.
	\end{proof}
	
	\section{Numerical Experiments}\label{Numeric}
	In this section, we present an implementation of the proposed \textbf{P}rojected \textbf{D}escent \textbf{S}ubgradient \textbf{M}ethod (\textbf{PDSM}) and evaluate its practical performance. The first experiment considers a collection of convex and nonconvex academic test problems and reports the main computational results. The second experiment addresses image denoising using the $\ell_1$ version of total variation. The third experiment concerns the approximation of the Pareto front for a nonconvex multiobjective optimization problem. The final experiment addresses a constrained data clustering problem in which the centroid of each cluster is required to lie within an octagon.
	
	We implemented the proposed \textbf{PDSM} in \textsc{Matlab} (R2022b) on a computer equipped with an Intel Core i5 processor and 16 GB of RAM. The convex quadratic subproblems arising in the computation of the search directions were solved using the \texttt{quadprog} solver. For problems with bound constraints, the projection admits a closed-form expression. For general polyhedral constraints, the projection subproblems were solved using an active-set method.
	
	In the subsequent experiments, all objective functions are locally Lipschitz. Moreover, most of them are weakly semismooth, being either piecewise linear or the pointwise maximum of finitely many smooth functions. Although verifying weak semismoothness may be difficult in general, our computational experience indicates that, for the locally Lipschitz functions considered in this section, the projected subgradient search of Algorithm~\ref{projected-mifflin-line search} terminates after finitely many iterations. In the rare event that the projected subgradient search fails to terminate, a small random perturbation of the current iterate is introduced as a heuristic to avoid stagnation.

	Regarding the parameters of the proposed method, we adopt the following choices:
	
	\begin{itemize}
		\item \textbf{Algorithm~\ref{exp-limited-line-search}.} The scale parameter is set to $\hat{t}:=0.005$, the sufficient decrease parameter to $\beta:=10^{-6}$, and $(S_{\max},p)$ is selected from the set
		\[
		\{(100,50),\,(150,100),\,(200,150),\,(500,300)\}.
		\]
		For challenging problems, the choice $(S_{\max},p)=(500,300)$ is recommended.
		\item \textbf{Algorithm~\ref{projected-mifflin-line search}.} The parameter $c$ is set to $0.9$, and the reduction factor is chosen as $r:=0.25$. Moreover, $t_i$ is updated by an interpolation procedure proposed in \cite{kiwielbook}.
		
		\item \textbf{Algorithm~\ref{main-alg}.} We set $\delta_0:=0.5$, $\varepsilon_0:=0.1$, and the reduction factor $\sigma:=0.5$. Moreover, the norm of $W(\tilde{t},k)$ is normalized throughout the optimization process by dividing it by
		\[
		\max\{1,\|\bm{x}_0\|,\|\boldsymbol{\xi}_0\|\}.
		\]
	\end{itemize}
	
	The termination criteria are specified separately for each experiment.

	%Regarding the parameters of the proposed method, our choices are as follows.  In Algorithm~\ref{exp-limited-line-search}, we work with the scale parameter $\hat{t}:=0.005$, and  sufficient decrease parameter $\beta:=10^{-6}$. Moreover, $(S_{\max}, p)$ can be chosen from the set $\{(100, 50), (150,100), (200, 150), (500, 300)\}$.
	% Regarding Algorithm~\ref{projected-mifflin-line search}, we put $c:=0.9$, and the reduction factor $r$ is set to be $0.25$.  In Algorithm~\ref{main-alg}, we work with $\delta_0:=0.5$, $\varepsilon_0:=0.1$, and the reduction factor $\sigma$ is set to be $0.5$. Moreover, the norm of $W(\tilde{t}, k)$ is normalized throughout the optimization process through division by $\max\{1, \lVert \bm x_0 \rVert, \lVert \boldsymbol{\xi}_0 \rVert\}$. Concerning the termination criteria, we will impose several conditions, which we specify individually for each experiment.

	\subsection{Academic Test Problems}
	
	In this experiment, we consider three classes of academic test problems according to the structure of their feasible regions, namely:
	\begin{itemize}
		\item[(i)] Bound Constrained Problems (BCP);
		
		\item[(ii)] Linearly Constrained Problems (LCP);
		
		\item[(iii)] General Constrained Problems (GCP), whose feasible regions are of the form
		\[
		\Omega:=\{\bm{x}\in\mathbb{R}^n:\; g_i(\bm{x})\le0,\quad i\in I\},
		\]
		where $I\subset\mathbb{N}$ is a finite index set and each
		$g_i:\mathbb{R}^n\rightarrow\mathbb{R}$ is a smooth function.
	\end{itemize}

	Although the proposed method is not designed to handle problem class~(iii) directly, it can be extended to general constrained problems through a sequential scheme based on linearizations of the constraint functions. This scheme is described in Algorithm~\ref{SLP}, where the proposed \textbf{PDSM} is applied to a sequence of linearly constrained subproblems. Developing termination criteria based on first-order optimality conditions, together with a convergence analysis, lies beyond the scope of this paper and constitutes an interesting direction for future research. In this algorithm, the stopping tolerances are set to $\varepsilon_1=\varepsilon_2:=5\times10^{-4}$. Furthermore, the nonnegative thresholds are chosen as
	$c_1:=10^{-2}$, $c_2:=0$, $\rho_1:=10^{-2}$, and $\rho_2:=10^{-3}$, while the initial trust-region radius is set to $\Delta_0:=0.1$.
	
	\begin{algorithm}[h]
		\caption{A Sequential Linearization Scheme}
		\label{SLP}
		\hspace*{\algorithmicindent}\textbf{Inputs:} Objective function $f:\mathbb{R}^n\to\mathbb{R}$, feasible region $\Omega:=\{\bm x\in\mathbb{R}^n\,:\, g_i(\bm x)\leq 0\,\, \text{for all}\,\, i\in I  \}$, initial radius $\Delta_0>0$, starting point $\bm x_0\in\Omega$, and positive stopping tolerances $\varepsilon_1$ and $\varepsilon_2$.  \\
		\hspace*{\algorithmicindent}\textbf{Parameters:} Nonnegative thresholds $c_1 >c_2\geq 0$ and $\rho_1>\rho_2>0$.   \\
		%	\hspace*{\algorithmicindent}\textbf{Outputs:} Step size $t_k\geq 0$, and indicator $I_k\in\{ 0, 1, 2\}$.\\
		
		\begin{algorithmic}[1]
			\STATE{\textbf{Initializations:} Set $k:=0$ ;}
			\WHILE{$True$}
			\STATE Set {\small $\Omega(\bm x_k):=\{\bm d\in\mathbb{R}^n\,:\, g_i(\bm x_k)+\nabla g_i(\bm x_k)^T \bm d \leq 0, \, \forall \,\, i\in I \,\, \text{and} \,\, \lVert \bm d \rVert_{\infty} \leq \Delta_k   \}    $} ;
			\STATE Employ \textbf{PDSM} to find the stationary point $\bm d^*_k$ of subproblem $$\min_{\bm d} f(\bm x_k + \bm d) \quad \text{s.t.} \quad \bm d\in \Omega(\bm x_k).  $$

			\IF{$\lvert f(\bm x_k + \bm d^*_k)-f(\bm x_k)  \rvert<\varepsilon_1$ \text{and}  $\max_{i\in I}\{ g_i(\bm x_k ), 0 \}< \varepsilon_2$  }
			\STATE{ Return $\bm x_k$  and \textbf{STOP} ;}
			\ELSIF {$ f(\bm x_k + \bm d^*_k)-f(\bm x_k) < -c_1$ \text{and}  $\max_{i\in I}\{ g_i(\bm x_k + \bm d^*_k), 0 \}< \rho_2$   }
			\STATE{Set $\bm x_{k+1}:=\bm x_k + \bm d^*_k$ \,\, \text{and} \,\, $\Delta_{k+1}:=2\Delta_k$ ; }
			\ELSIF {$ f(\bm x_k + \bm d^*_k)-f(\bm x_k) < -c_2$ \text{and}  $\max_{i\in I}\{ g_i(\bm x_k + \bm d^*_k), 0 \}< \rho_1$   }
			\STATE{Set $\bm x_{k+1}:=\bm x_k + \bm d^*_k$ \text{and} \,\, $\Delta_{k+1}:=\Delta_k$ ; }
			\ELSE
			\STATE{Set $\bm x_{k+1}:=\bm x_k$ \,\, \text{and} \,\, $\Delta_{k+1}:=0.5\Delta_k$ ; }
			\ENDIF
			\STATE{Set $k\leftarrow k+1$ ;}

			\ENDWHILE
		\end{algorithmic}
		%\hspace*{\algorithmicindent}\textbf{ End Function}
	\end{algorithm}

	We will consider the following three instances of general constrained minimization problems in our experiments:
	
	\begin{align}\tag{\textbf{P1}}
		\begin{split}
			&\quad\quad\,\,\min\,\, \ln(e^{\lvert x_1 \rvert}+e^{\lvert x_2 \rvert}+e^{\lvert x_3 \rvert})\\&
			\text{s.t.} \quad \sinh(x_1)+\cosh(x_2)-x_3^2-1\leq 0.
		\end{split}
	\end{align}
	The global minimizer is $\bm x^*=\bm 0$ with optimal value $f(\bm x^*)=1.0986$. The suggested starting point is $\bm x_0=[-10, 2, 2]$.
	The second problem is
	{\small
		\begin{align}\tag{\textbf{P2}}
			\begin{split}
				&\quad\quad\,\, \min\,\, \max\{x_1^2-2x_1+(x_2-1)^2-3x_2,  -x_1^2+2x_1-(x_2-1)^2+4x_2 +13 \}\\&
				\text{s.t.} \quad -\ln(x_1)-x_2^2+1\leq 0,\\&
				\qquad \,\,-x_1+1\leq 0,\\&
				\qquad \,\,\,\,\,x_2+1\leq 0.
			\end{split}
		\end{align}
	}
	A (local) minimizer is given by $\bm x^*=[1, -1]$  with $f(\bm x^*)=6$. The suggested starting point is $\bm x_0=[10, -10]$. The third problem has a piecewise linear objective function and is defined by
	\begin{align}\tag{\textbf{P3}}
		\begin{split}
			&\quad\quad\,\,\min\,\, \max\{x_1, x_2, x_3\} + x_1+x_2+x_3\\&
			\text{s.t.} \quad -e^{\frac{x_1+x_2+x_3}{3^{\textcolor{white}{1}}       }}+e\leq 0,
		\end{split}
	\end{align}
	whose global minimizer is  $\bm x^*=[1, 1, 1]$ with optimal value $f(\bm x^*)=4$. The suggested starting point is $\bm x_0=[3, 3, 3]$. 
	
	Table~\ref{Table1} summarizes the test problems considered in this experiment, including BCP, LCP, and GCP. Here, $n$ denotes the problem dimension, and $f^*$ is a known (local) optimal value. For the BCP class, lower and upper bounds, denoted by \textbf{LB} and \textbf{UB}, are imposed on all variables so that the known (local) minimizer is neither an interior point of the feasible region nor a differentiable point of the objective function. The only exception is problem P11, for which a local minimizer is unavailable. A detailed description of test problems P1--P22 can be found in \cite{Bagirov2014,kiwielbook}.
	\begin{table}[h]
		\centering
		\caption{A list of test problems}\label{Table1}
		\resizebox{\textwidth}{!}{%
			\begin{tabular}{|lllllllllclllll|}
				\hline
				&  \rule{0pt}{3ex}P  &  & Name                   &  & $n$ &  & $f^*$     &  & Convex? &  & Class &  & [\textbf{LB}, \textbf{UB} ]          &  \\ \cline{2-2} \cline{4-4} \cline{6-6} \cline{8-8} \cline{10-10} \cline{12-12} \cline{14-14}
				&  \rule{0pt}{3ex}1  &  & MXHILB                 &  & 200 &  & 0.0000    &  & Yes       &  & BCP   &  & $[\bm 0, \bm {2} ]$  &  \\
				& 2  &  & L1HILB                 &  & 200 &  & 0.0000  &  & Yes       &  & BCP   &  & $[\bm 0, \bm{2}]$   &  \\
				& 3  &  & MAXL                   &  & 200 &  & 1.0000    &  & Yes       &  & BCP   &  & $[\bm 1, \bm 3]$    &  \\
				& 4  &  & MAXQ                   &  & 200 &  & 1.0000    &  & Yes       &  & BCP   &  & $[\bm 1, \bm 3]$    &  \\
				& 5  &  & Chained LQ             &  & 200 &  & -281.4284  &  & Yes       &  & BCP   &  & $[\bm 1/\sqrt{\bm 2}, \bm 5]$    &  \\
				& 6  &  & Chained CB3 I          &  & 200 &  & 398.0000 &  & Yes       &  & BCP   &  & $[\bm 1, \bm 3]$    &  \\
				& 7  &  & Chained CB3 II         &  & 200 &  & 398.0000 &  & Yes       &  & BCP   &  & $[\bm 1, \bm 3]$    &  \\
				& 8  &  & Number of Active Faces &  & 200 &  & 0.0000    &  & No        &  & BCP   &  & $[\bm 0, \bm 2]$    &  \\
				& 9  &  & Chained Crescent 1     &  & 200 &  & 0.0000  &  & No        &  & BCP   &  & $[\bm{-2}, \bm 0]$    &  \\
				& 10 &  & Chained Crescent 2     &  & 200 &  & 0.0000  &  & No        &  & BCP   &  & $[\bm{-2}, \bm 0]$    &  \\
				& 11 &  & Chained Mifflin 2      &  & 200 &  & -140.8600 &  & No        &  & BCP   &  & $[\bm {-1}, \bm 1]$   &  \\
				& 12 &  & Brown Function 2         &  & 200 &  & 0.0000  &  & No        &  & BCP   &  & $[\bm 0, \bm 1]$   &  \\
				& 13 &  & Rosenbrock             &  & 2   &  & 1.0000    &  & No        &  & BCP   &  & $[\bm {-10}, \bm 0 ]$  &  \\
				& 14 &  & Wong 2C               &  & 10  &  & 24.3062   &  & Yes        &  & LCP   &  & \multicolumn{1}{l}{-} &  \\
				& 15 &  & Ill-conditioned LP                &  & 15  &  & -20.0420  &  & Yes        &  & LCP   &  & \multicolumn{1}{l}{-} &  \\
				& 16 &  & MAD 1                  &  & 2   &  & -0.3896   &  & No        &  & LCP   &  & \multicolumn{1}{l}{-} &  \\
				& 17 &  & MAD 2                  &  & 2   &  & -0.3303   &  & No        &  & LCP   &  & \multicolumn{1}{l}{-} &  \\
				& 18 &  & MAD 4                  &  & 2   &  & -0.4489   &  & No        &  & LCP   &  & \multicolumn{1}{l}{-} &  \\
				& 19 &  & MAD 5                  &  & 2   &  & -0.4292   &  & No        &  & LCP   &  & \multicolumn{1}{l}{-} &  \\
				& 20 &  & Pentagon               &  & 6   &  & -1.8596   &  & No        &  & LCP   &  & \multicolumn{1}{l}{-} &  \\
				& 21 &  & MAD 6                  &  & 7   &  & 0.040152    &  & No        &  & LCP   &  & \multicolumn{1}{l}{-} &  \\
				& 22 &  & MAD 8                  &  & 20  &  & 0.5069    &  & No        &  & LCP   &  & \multicolumn{1}{l}{-} &  \\
				& 23 &  & $\bm{P1}$ in this paper       &  & 3   &  & 1.0986     &  & No        &  & GCP   &  & \multicolumn{1}{l}{-} &  \\
				& 24 &  & $\bm{P2}$ in this paper       &  & 2   &  & 6.0000    &  & No        &  & GCP   &  & \multicolumn{1}{l}{-} &  \\
				& 25 &  & $\bm{P3}$ in this paper       &  & 3   &  & 4.0000    &  & No        &  & GCP   &  & \multicolumn{1}{l}{-} &  \\ \hline
			\end{tabular}
		}
	\end{table}

	For problems P1--P22, since the optimal value $f^*$ is available, the optimization process is terminated once the relative error
	\begin{equation}\label{RE}
		RE(\bm x_k):= \frac{\lvert f(\bm x_k)-f^* \rvert}{\rvert f^*\lvert + 1},
	\end{equation}
	falls below $5\times10^{-4}$. In addition, the maximum number of iterations is set to $10^4$.
	
	Table~\ref{Table2} reports the numerical results obtained by the proposed \textbf{PDSM} on problems P1--P22. Here, ``Iter'' denotes the total number of iterations, while ``Fun'' and ``Sub'' represent the numbers of function and subgradient evaluations, respectively. Furthermore, $f_{\mathrm{best}}$ is the lowest objective value achieved during the optimization process, $v_f$ denotes the value of the optimality certificate $v_k$ at the final iteration, and ``RE'' is the relative error associated with $f_{\mathrm{best}}$. As shown in the table, the proposed method attained the prescribed accuracy for all test problems within a reasonable computational time. It is also worth noting that the number of subgradient evaluations is substantially smaller than the number of function evaluations. This behavior stems from the fact that the \textbf{PDSM} computes a new nonredundant subgradient only when the limited exponential line search of Algorithm~\ref{exp-limited-line-search} indicates that the current search direction is not an effective direction. In some cases, the number of function or subgradient evaluations is smaller than the number of iterations. This occurs because, once the first conditional block of  Algorithm~\ref{main-alg} is executed, the iteration counter $k$ is increased by one without performing any function or subgradient evaluations.

	% Concerning problems P23-P25, at each iteration, Algorithm~\ref{SLP} executes \textbf{PDSM} using optimality tolerance $\tau:=5*10^{-2}$, and maximum number of iterations $10^2$.

	\begin{table}[h]
		\centering
		\caption{Numerical results of the \textbf{PDSM} on problems P1-P22.}\label{Table2}
		\resizebox{\textwidth}{!}{%
			\begin{tabular}{|lllllllllllllllll|}
				\hline
				& \rule{0pt}{3ex} P  &  & Iter &  & Fun    &  & Sub  &  & $f_{best}$ &  & $v_f$  &  & RE     &  & Time(s) &  \\ \cline{2-2} \cline{4-4} \cline{6-6} \cline{8-8} \cline{10-10} \cline{12-12} \cline{14-14} \cline{16-16}
				& \rule{0pt}{3ex}1  &  & 263  &  & 601   &  & 351  &  & 0.0004     &  & 0.0002 &  & 0.0004 &  & 11.12    &  \\
				& 2  &  & 39   &  & 66     &  & 34   &  & 0.0000     &  & 0.0121 &  & 0.0000 &  & 0.70     &  \\
				& 3  &  & 241  &  & 835   &  & 235  &  & 1.0009     &  & 0.0100 &  & 0.0004 &  & 0.25     &  \\
				& 4  &  & 365 &  & 2328  &  & 360 &  & 1.0009     &  & 0.0304 &  & 0.0004 &  & 0.32    &  \\
				& 5  &  & 53   &  & 724  &  & 47   &  & -281.2880  &  & 0.0152 &  & 0.0004 &  & 1.29     &  \\
				& 6  &  & 31  &  & 610   &  & 28  &  & 398.0009  &  & 0.0707 &  & 0.0000 &  & 1.40     &  \\
				& 7  &  & 53   &  & 1126   &  & 248  &  & 398.1516  &  & 0.0309 &  & 0.0003 &  & 1.57     &  \\
				& 8  &  & 26   &  & 30     &  & 16   &  & 0.0000     &  & 0.0007 &  & 0.0000 &  & 0.04     &  \\
				& 9  &  & 46   &  & 171   &  & 136  &  & 0.0004     &  & 0.0009 &  & 0.0004 &  & 7.36    &  \\
				& 10 &  & 39   &  & 268   &  & 34   &  & 0.0000     &  & 0.0196 &  & 0.0000 &  & 0.44     &  \\
				& 11 &  & 30 &  & 605 &  & 24 &  & -140.8114  &  & 0.0120 &  & 0.0003 &  & 0.33   &  \\
				& 12 &  & 25  &  & 464   &  & 19  &  & 0.0004     &  & 0.0141 &  & 0.0004 &  & 1.19   &  \\
				& 13 &  & 14   &  & 20     &  & 11   &  & 1.0000     &  & 0.0891 &  & 0.0000 &  & 0.01     &  \\
				& 14 &  & 458  &  & 131029  &  & 1053  &  & 24.3188    &  & 0.0237 &  & 0.0004 &  & 2.92     &  \\
				& 15 &  & 36  &  & 60 &  & 31  &  & -20.0315  &  & 0.0235 &  & 0.0004 &  & 0.03    &  \\
				& 16 &  & 40  &  & 1857  &  & 263  &  & -0.3891    &  & 0.0070 &  & 0.0003 &  & 0.06     &  \\
				& 17 &  & 13   &  & 6      &  & 4    &  & -0.3296    &  & 0.0019 &  & 0.0004 &  & 0.01     &  \\
				& 18 &  & 62   &  & 102    &  & 52   &  & -0.4484    &  & 0.0009 &  & 0.0002 &  & 0.03     &  \\
				& 19 &  & 40   &  & 56     &  & 29   &  & -0.4289    &  & 0.0002 &  & 0.0001 &  & 0.02     &  \\
				& 20 &  & 307  &  & 28348 &  & 1000 &  & -1.8581   &  & 0.0045 &  & 0.0004 &  & 0.74    &  \\
				& 21 &  & 85  &  & 3193  &  & 276  &  & 0.0406     &  & 0.0019 &  & 0.0004 &  & 0.32    &  \\
				& 22 &  & 18  &  & 758 &  & 12  &  & 0.5076     &  & 0.0132 &  & 0.0004 &  & 0.05  &  \\ \hline
			\end{tabular}
		}
	\end{table}

	Next, we apply Algorithm~\ref{SLP} to the general constrained problems P23--P25. At each iteration, the algorithm invokes the proposed \textbf{PDSM} with an optimality tolerance of $\tau:=5\times10^{-2}$ and a maximum of 50 iterations. The numerical results are reported in Table~\ref{Table3}, where ``Con.Acc'' denotes the value of
	\[
	\max_{i\in I}\{g_i(\bm{x}),0\},
	\]
	at the final iteration, which measures the constraint violation. Although the proposed sequential scheme has been tested on only a limited number of general constrained problems, the results reported in Table~\ref{Table3} suggest that it provides a promising basis for developing optimization methods that preserve the original form of the objective function while approximating the feasible region through successive linearizations of the constraints.

	\begin{table}[h]
		\centering
		\caption{Numerical results of Algorithm~\ref{SLP} on problems P23-P25.}\label{Table3}
		\resizebox{\textwidth}{!}{%
			\begin{tabular}{|lclllllllllllllll|}
				\hline
				&\rule{0pt}{3ex} P                     &  & Iter &  & Fun &  & Sub &  & $f_{best}$ &  & Con.Acc &  & RE &  & Time(s) &  \\ \cline{2-2} \cline{4-4} \cline{6-6} \cline{8-8} \cline{10-10} \cline{12-12} \cline{14-14} \cline{16-16}
				&\rule{0pt}{3ex}23                     &  &   15   &  &   9271   &  &   1100  &  &    1.0996     &  &  3E-5          &  &  4E-4  &  &  0.32       &  \\
				& 24                     &  &  11    &  & 46    &  &  35   &  &  6.0000        &  &    1E-16      &  &  1E-8  &  & 0.09        &  \\
				& 25 &  &  8    &  &14166     &  &   2009  &  &    4.0002     &  &     0     &  &  4E-5  &  &   0.41      &  \\ \hline
			\end{tabular}
		}
	\end{table}
	
	To the best of our knowledge, projection-based descent methods for minimizing weakly semismooth functions over closed and convex sets have received limited attention in the literature. For comparison purposes, we consider the classical \textbf{P}rojected \textbf{S}ubgradient \textbf{M}ethod (\textbf{PSM}) and its generalization, the \textbf{M}irror \textbf{D}escent \textbf{M}ethod
	\footnote{For the mirror descent method, the Bregman distance is generated by the negative entropy.}
	(\textbf{MDM}), both equipped with the adaptive step size rule proposed in \cite{Amir_Beck_first_order}. Although these methods are simple to implement, they are not descent methods, and their convergence theory is restricted to convex objective functions.

	\begin{figure}[!h]
		\centering % <-- added
		
		\begin{subfigure}
			{\includegraphics[width=.49\textwidth]{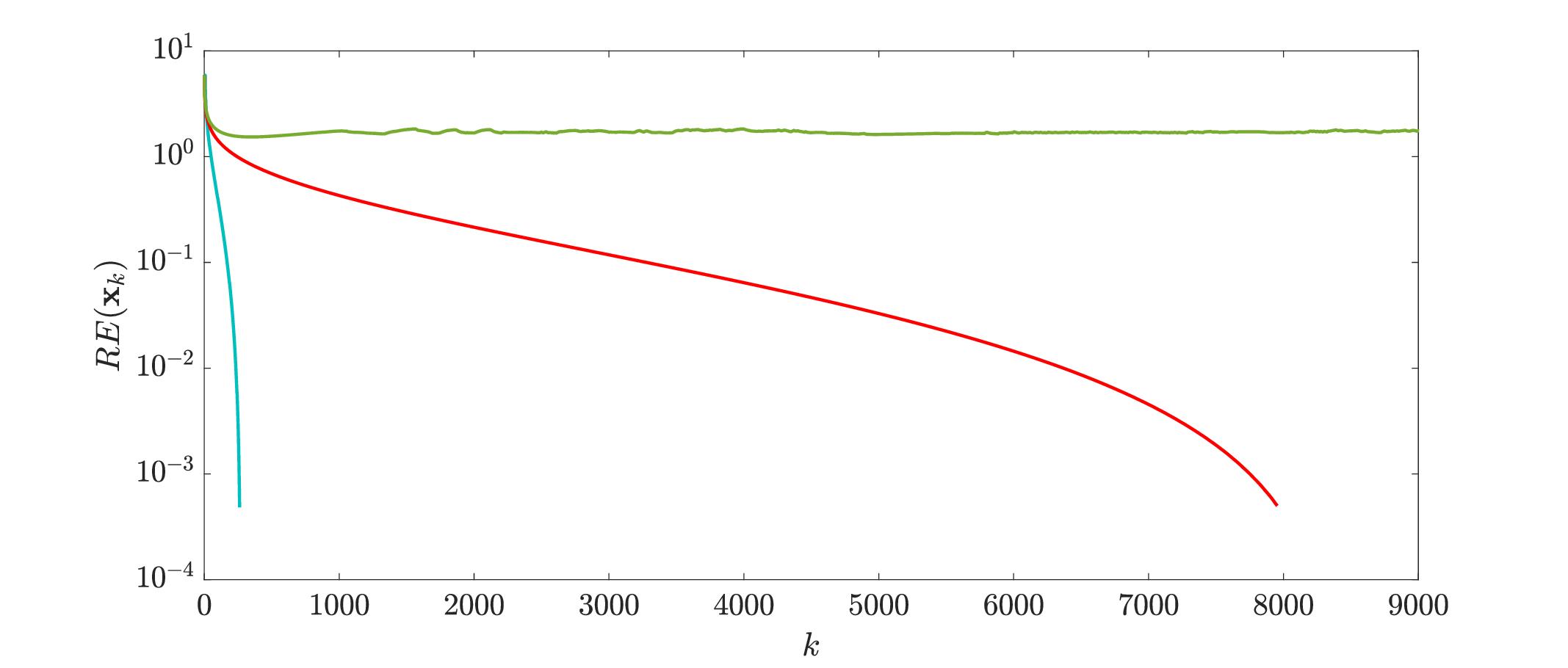}}
			{\includegraphics[width=.49\textwidth]{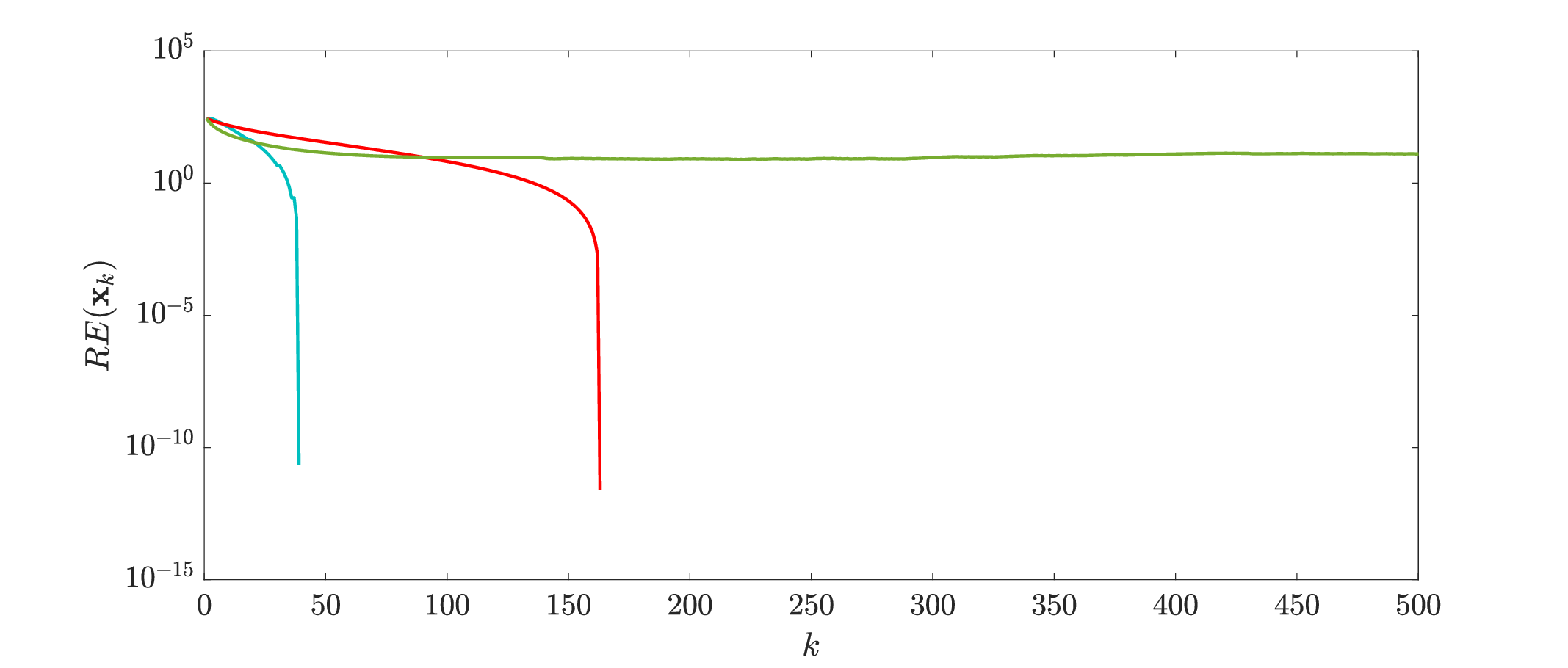}}
		\end{subfigure}\hfill
		\medskip
		\begin{subfigure}
			{\includegraphics[width=.49\textwidth]{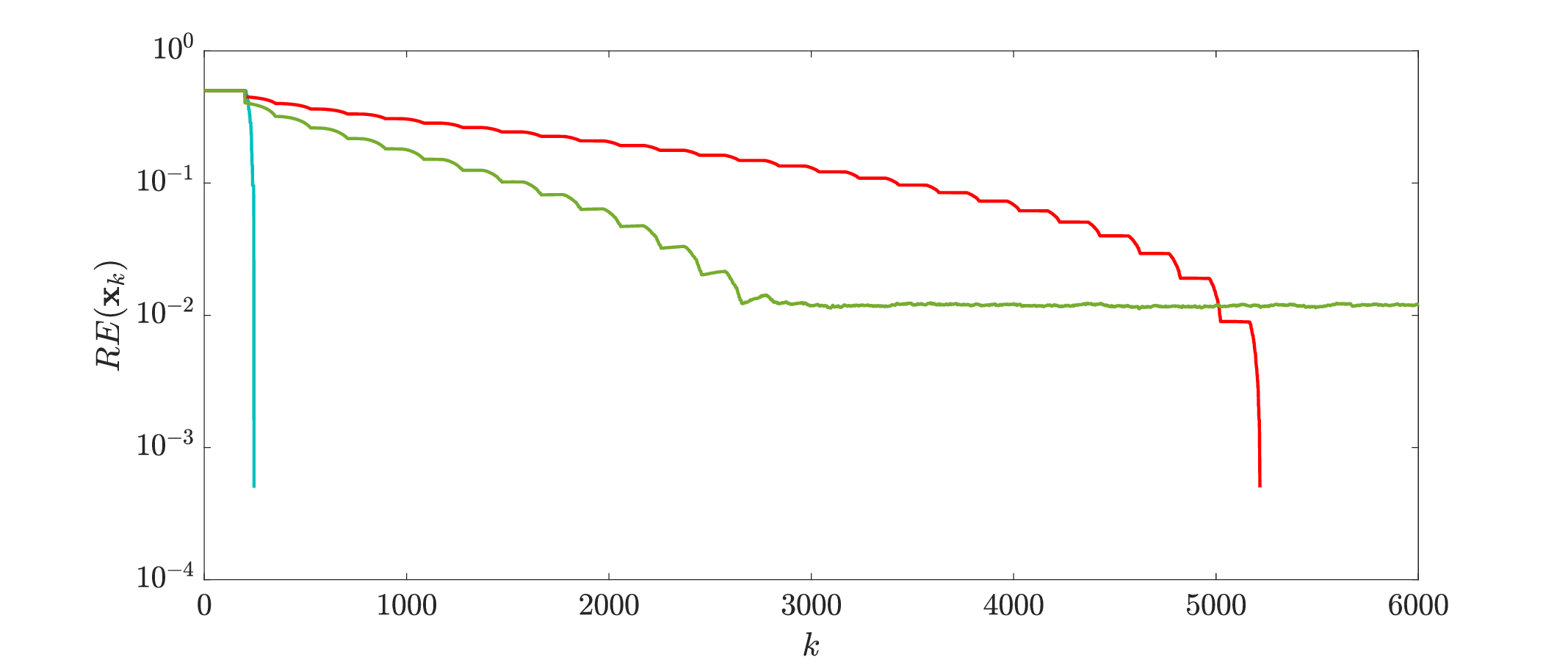}}
			{\includegraphics[width=.49\textwidth]{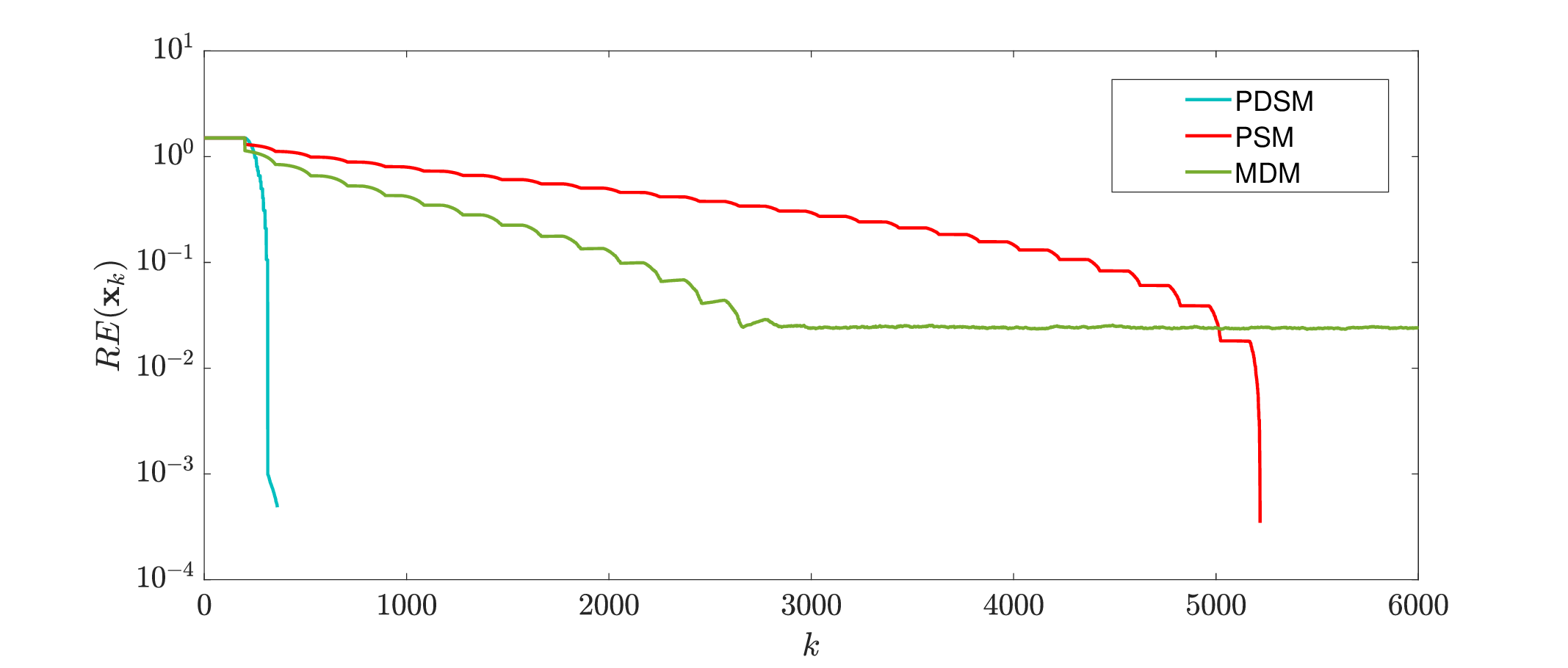}}
		\end{subfigure}
		
		\caption{Relative error versus iteration number for \textbf{PDSM}, \textbf{PSM}, and \textbf{MDM} on test problems P1 (top left), P2 (top right), P3 (bottom left), and P4 (bottom right).}
		\label{Fig1}
	\end{figure}

	Figure~\ref{Fig1} illustrates the performance of \textbf{PDSM}, \textbf{PSM}, and \textbf{MDM} on the convex test problems P1--P4. In each run, an algorithm was terminated either when the relative error defined in \eqref{RE} fell below $5\times10^{-4}$ or when the number of iterations reached $2\times10^{4}$. As shown, both \textbf{PDSM} and \textbf{PSM} achieved the prescribed accuracy. However, \textbf{PSM} required substantially more iterations than \textbf{PDSM} to do so. In contrast, \textbf{MDM} did not achieve the prescribed accuracy before reaching the maximum number of iterations. Overall, \textbf{PDSM} and \textbf{PSM} substantially outperformed \textbf{MDM} in terms of solution accuracy.
	
	\subsection{Image Denoising}
	As a large-scale application of the proposed method, we consider the problem of image denoising.
	Let $\bm A\in\mathbb{R}^{m\times n}$ be a normalized noisy image. A common approach to restoring the original image from the noisy image $\bm A$ is to solve the following optimization problem based on the $\ell_1$-variant of the total variation regularization:
	\begin{align}\label{cameraman}
		\begin{split}
			\min_{\bm U} \,\, &\lVert \bm U- \bm A \rVert_F^2 + \mu \left\{ \sum_{j=1}^{n} \sum_{i=1}^{m-1} \lvert U_{i+1,j}-U_{i,j} \rvert + \sum_{i=1}^{m} \sum_{j=1}^{n-1} \lvert U_{i,j+1}-U_{i,j} \rvert  \right\}\\&
			\text{s.t.} \,\,\, 0\leq U_{i,j}\leq 1, \,\,\,\, \forall \,\, i,j,
		\end{split}
	\end{align}
	where $\lVert \cdot \rVert_F $ denotes the Frobenius norm and $\mu>0$ is a  positive  regularization parameter. In this convex optimization model, the second term of the objective function, known as the total variation regularization, suppresses vertical and horizontal intensity variations, while the first term encourages the restored image to remain close to the observed noisy image, thereby balancing noise suppression and feature preservation.
	
	Let the \emph{Cameraman} image of size $640\times640$ be the original image. The noisy images were generated by adding zero-mean uniformly distributed random noise with noise levels of 0.05, 0.10, and 0.15 to the original image. Since each pixel corresponds to a decision variable, the resulting optimization problem has $640\times640=409,\!600$ variables, making it a large-scale optimization problem. We then applied the proposed \textbf{PDSM} to solve problem~\eqref{cameraman}. A randomly generated starting point, an optimality tolerance of $\tau:=0.1$, and a regularization parameter of $\mu:=0.07$ were used throughout the experiment.

	The quality of the restored images is evaluated using the \emph{Peak Signal-to-Noise Ratio} (PSNR) and the \emph{Structural Similarity Index Measure} (SSIM) \cite{PSNR}. PSNR is a widely used measure of image fidelity based on the pixel-wise reconstruction error, with higher values indicating better restoration quality. SSIM measures the similarity between two images in terms of their luminance, contrast, and structural information. Its value ranges from 0 to 1, with values closer to 1 indicating greater structural similarity to the original image.

	\begin{table}[h]
		\centering
		\caption{Numerical results of the \textbf{PDSM} on the considered instance of image denoising problem \eqref{cameraman}.}\label{Table4}
		\resizebox{\textwidth}{!}{%
			\begin{tabular}{|lllllllllllllllllllllll|}
				\hline
				& \rule{0pt}{3ex}$\sigma$ &  & Iter &  & Fun  &  & Sub  &  & $f_{best}$ &  & $v_f$  &  & Time(s) &  & PSNR(n) &  & PSNR(r) &  & SSIM(n) &  & SSIM(r) &  \\ \cline{2-2} \cline{4-4} \cline{6-6} \cline{8-8} \cline{10-10} \cline{12-12} \cline{14-14} \cline{16-16} \cline{18-18} \cline{20-20} \cline{22-22}
				& \rule{0pt}{3ex}0.05     &  & 746  &  & 6189 &  & 743  &  & 650.7528   &  & 0.0619 &  & 74.91   &  & 30.8224     &  & 32.1402     &  & 0.6718      &  & 0.8875      &  \\
				& 0.10     &  & 809  &  & 806  &  & 9999 &  & 1096.4137  &  & 0.0410 &  & 115.36  &  & 24.9816     &  & 31.5998     &  & 0.4045      &  & 0.8815      &  \\
				& 0.15     &  & 647  &  & 5835 &  & 644  &  & 1806.1525  &  & 0.0315 &  & 71.22   &  & 21.5781     &  & 30.6823     &  & 0.2737      &  & 0.8679      &  \\ \hline
			\end{tabular}
		}
	\end{table}

	Table~\ref{Table4} reports the numerical performance of the proposed \textbf{PDSM} at three noise levels, $\sigma\in\{0.05,0.1,0.15\}$. In this table, PSNR(n) denotes the PSNR of the noisy image with respect to the original image, whereas PSNR(r) denotes the PSNR of the reconstructed image with respect to the original image. Similarly, SSIM(n) and SSIM(r) denote the corresponding SSIM values for the noisy and reconstructed images, respectively.
	
	The results in Table~\ref{Table4} demonstrate the effectiveness of the optimization process carried out by \textbf{PDSM} for the considered image denoising problem. In all three cases, the reconstructed images exhibit substantial improvements in both PSNR and SSIM compared with the corresponding noisy images, with the improvement becoming more pronounced as the noise level increases.  Moreover, we observed that the majority of the computational time was spent verifying the optimality condition $v_k\leq \tau$. This observation suggests considering alternative termination criteria in this context, such as detecting negligible progress in the objective function values over a prescribed number of serious iterations. 
	In addition to the results reported in Table~\ref{Table4}, Figure~\ref{Fig2} presents the original, noisy, and reconstructed images for the noise level $\sigma=0.15$.

	\begin{figure}
		\centering
		
		\begin{subfigure}
			{\includegraphics[width=.32\textwidth]{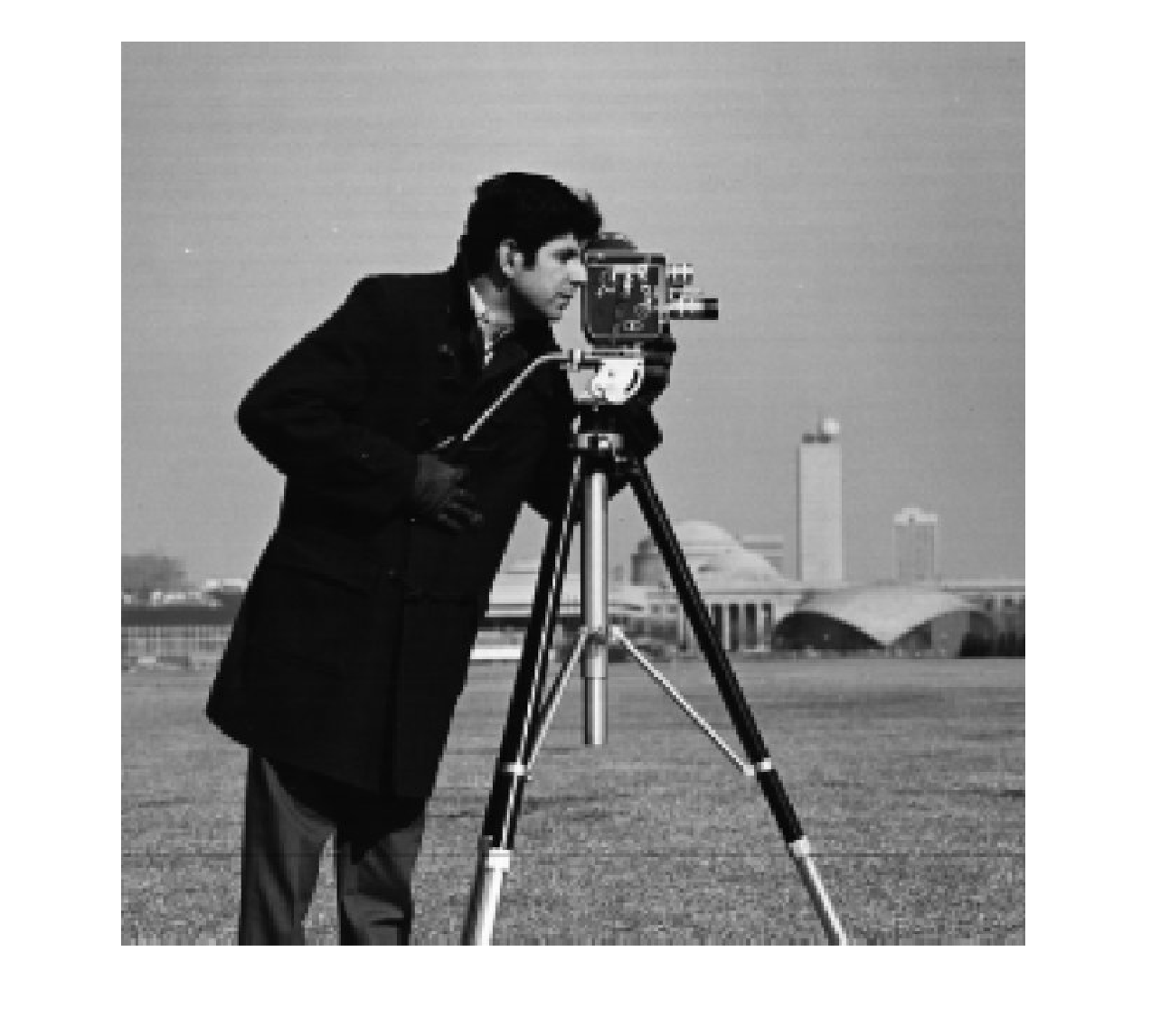}}
			{\includegraphics[width=.32\textwidth]{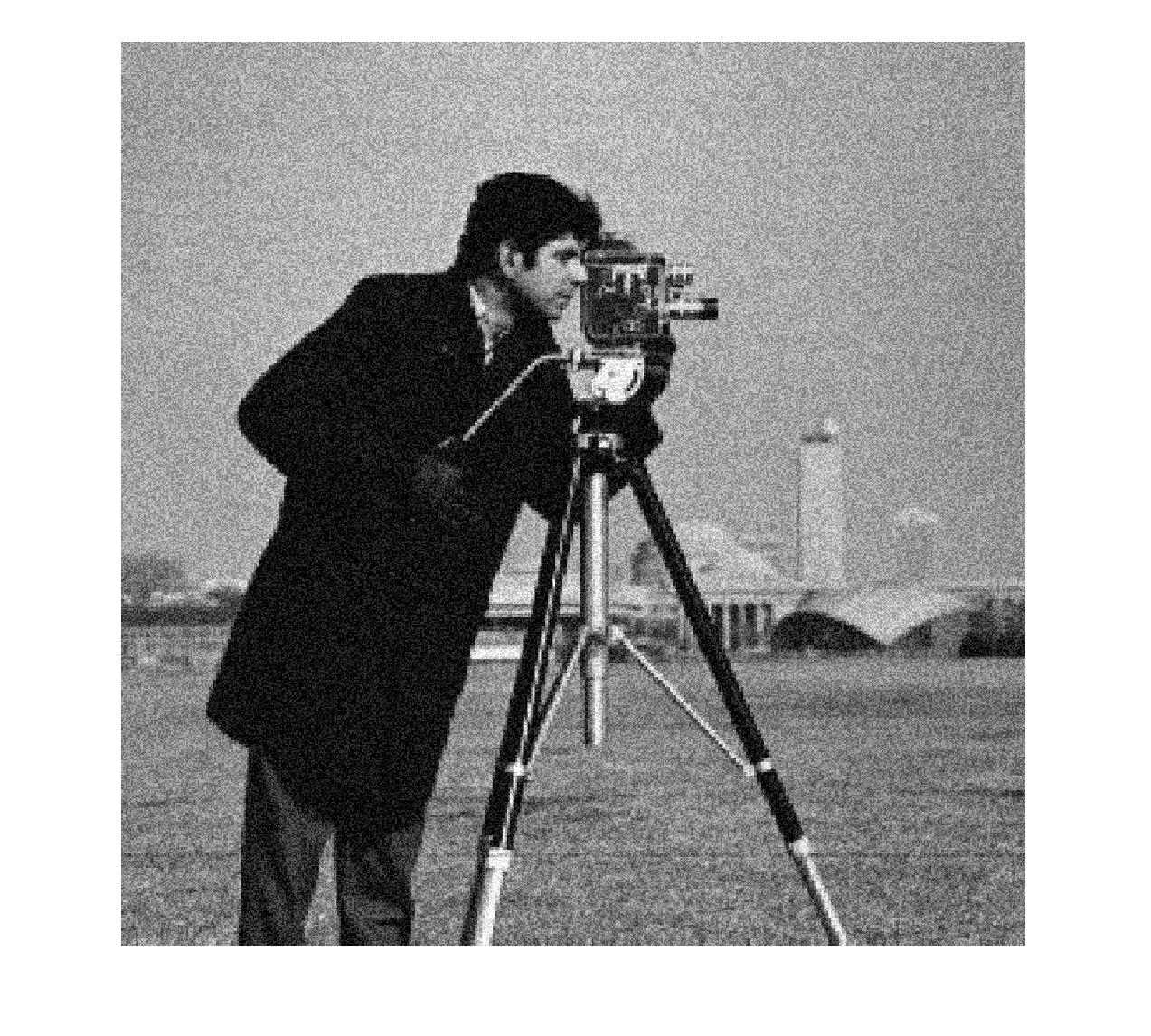}}
			{\includegraphics[width=.32\textwidth]{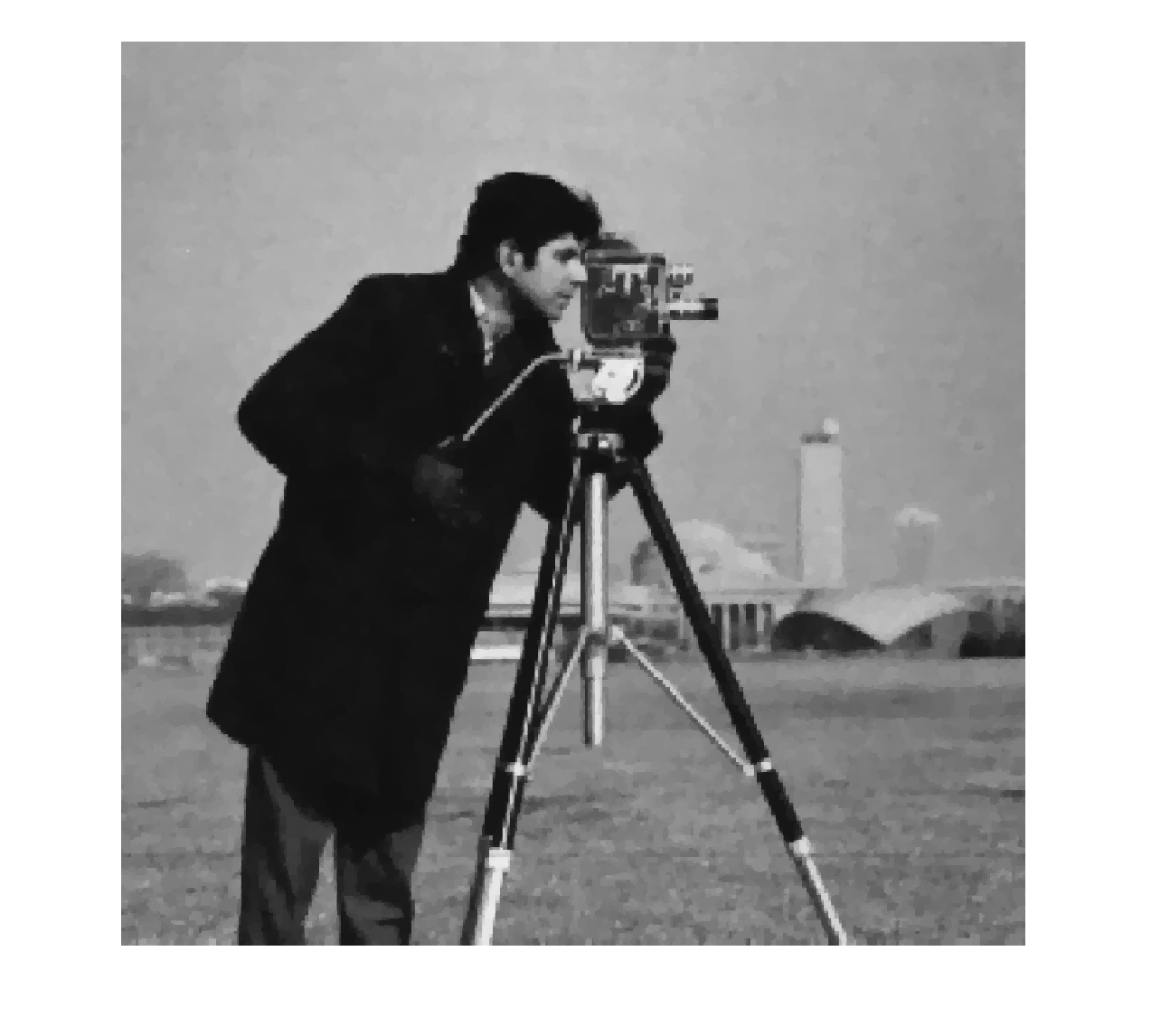}}
			
		\end{subfigure}\hfill
		\caption{Left. the original image. Middle. the noisy image using noise level $\sigma=0.15$. Right.  the reconstructed image by the \textbf{PDSM}. }\label{Fig2}
	\end{figure}

	\subsection{Multiobjective Optimization}
	In this experiment, we consider the following nonsmooth nonconvex constrained bi-objective optimization problem:
	\begin{align}\label{bi-objective}
		\min \,\, \big(f_1(\bm x), f_2(\bm x)\big) \quad \text{s.t.} \quad -0.5x_1+x_2\leq 0,
	\end{align}
	where $f_1, f_2:\mathbb{R}^2\to \mathbb{R}$ are given by
	\begin{equation*}
		f_1(\bm x):= \max\{ x_1^2+(x_2-1)^2+x_2-1, -x_1^2-(x_2-1)^2+x_2+1  \},
	\end{equation*}
	and
	\begin{equation*}
		f_2(\bm x):= \max\{ -x_1-x_2, -x_1-x_2+x_1^2+x_2^2-1 \},
	\end{equation*}
	respectively. This bi-objective optimization problem combines the \emph{Crescent} and \emph{LQ} test functions \cite{Bagirov2014}. The objective space  and the corresponding  \emph{Pareto front} of this problem are shown in the left and middle plots of Figure~\ref{Fig3}, respectively. Since the objective space is $\mathbb{R}^2_{\geq}$-convex, the weighted sum method can be used to approximate the entire Pareto front \cite{Ehrgott2005}. Accordingly, let $\Delta_2$ be the two-dimensional unit simplex, i.e.,
	$$\Delta_2:=\{(\lambda_1, \lambda_2)\in\mathbb{R}^2 \,\,:\,\, \lambda_1+ \lambda_2=1, \,\, \lambda_1\geq 0, \,\, \lambda_2\geq 0  \},$$
	and for a weighting vector $\boldsymbol{\lambda}\in\Delta_2$, we consider the following single-objective weighted sum problem:
	\begin{equation}
		\min \,\, \lambda_1 f_1(\bm x) + \lambda_2 f_2(\bm x) \qquad \text{s.t.} \quad -0.5 x_1 + x_2\leq 0.
	\end{equation}
	It is well known that any optimal solution of the weighted sum problem is a Pareto point of the bi-objective problem. Conversely, for an appropriate choice of the weighting vector $\boldsymbol{\lambda}\in\Delta_2$, every Pareto point of the bi-objective problem is an optimal solution of the corresponding weighted sum problem. In this respect, for a given $m\in\mathbb{N}$, we consider the following uniform grid of the simplex $\Delta_2$:
	\begin{equation}
		\Lambda_m:=\left\{(\lambda_1^i,\lambda_2^i):\,
		\lambda_1^i=\frac{i}{m},\;
		\lambda_2^i=1-\frac{i}{m},\;
		i=0,1,\ldots,m
		\right\}.
	\end{equation} 
	For each $\boldsymbol{\lambda}\in\Lambda_{200}$, we solved the corresponding weighted sum problem using the \textbf{PDSM} with the feasible starting point $\bm x_0=[8,4]$ and the optimality tolerance $\tau:=10^{-4}$. The right plot of Figure~\ref{Fig3} shows the resulting approximation of the Pareto front.

	\begin{figure}
		\centering
		
		\includegraphics[width=\textwidth]{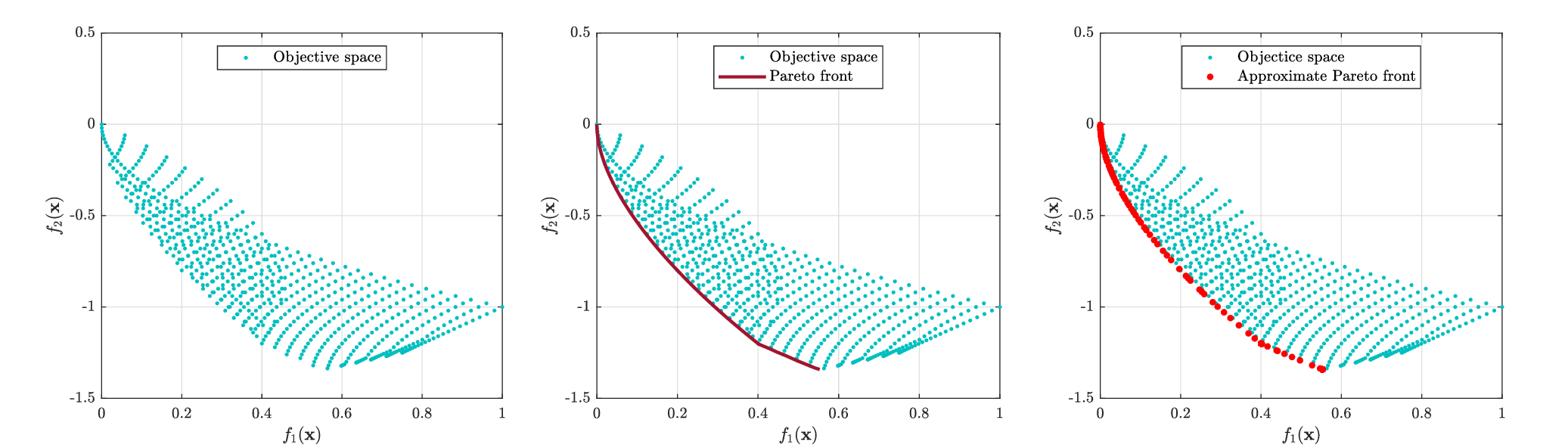}
		\caption{Left. the objective space. Middle. the Pareto front. Right. obtained approximate Pareto front by the \textbf{PDSM}. }\label{Fig3}
	\end{figure}

	When the Pareto front of a bi-objective optimization problem is connected, the \emph{Hole Absolute Size} (HAS) metric \cite{Multi-Indicators} provides an effective measure for evaluating the uniformity of the distribution of solutions along an approximate Pareto front.
	Let $\mathcal{P}$ denote an approximation of the Pareto front, with its elements sorted in ascending order according to the first objective function. For \mbox{$j = 1, \ldots, \lvert \mathcal{P}\rvert-1$}, define
	$d_j := \lVert \mathbf{p}_j - \mathbf{p}_{j+1} \rVert$,
	where $\mathbf{p}_j, \mathbf{p}_{j+1} \in \mathcal{P}$ are two consecutive solutions. 
	%Let
	%$$
	%\mu: = \frac{1}{\lvert \mathcal{P}\rvert-1}\sum_{j=1}^{\lvert \mathcal{P}\rvert-1} d_j
	%$$
	%be the average distance between consecutive solutions. 
	The HAS  indicator is then defined as
	$$
	\mathrm{HAS}(\mathcal{P}): =
	\max_{1 \le j \le |\mathcal{P}|-1} d_j.
	$$
	The HAS metric quantifies the absolute size of the largest gap between consecutive solutions,  with smaller value indicating a more evenly distributed approximation of the Pareto front.
	
	Table \ref{Table5} presents the computational performance of the \textbf{PDSM} in approximating the Pareto front of bi-objective problem~\eqref{bi-objective} using $\Lambda_{50}, \Lambda_{100}, \Lambda_{150}$, and $\Lambda_{200}$. As one would expect, we observe a downward trend in the HAS metric as $m$ increases, but at the cost of increased computational time.
	
	\begin{table}[]
		\centering
		\caption{Computational performance of the \textbf{PDSM} to approximate the Pareto front of bi-objective problem \eqref{bi-objective}.}\label{Table5}
		\resizebox{\textwidth}{!}{%
			\begin{tabular}{|lllllllllll|}
				\hline
				&\rule{0pt}{3ex}$m$ &  & Fun    &  & Sub   &  & HAS    &  & Time(s) &  \\ \cline{2-2} \cline{4-4} \cline{6-6} \cline{8-8} \cline{10-10}
				&\rule{0pt}{3ex}50  &  & 223395 &  & 10322 &  & 0.1585 &  & 16.32   &  \\
				& 100 &  & 452928 &  & 22865 &  & 0.0838 &  & 38.32   &  \\
				& 150 &  & 702250 &  & 33498 &  & 0.0638 &  & 59.85   &  \\
				& 200 &  & 905116 &  & 44133 &  & 0.0578 &  & 69.01   &  \\ \hline
			\end{tabular}
		}
	\end{table}
	
	\subsection{Data Clustering}
	
	For some $m\in\mathbb{N}$, let
	$
	\bm A=\{\bm a_1,\bm a_2,\ldots,\bm a_m\}\subset\mathbb{R}^{n}
	$
	be a finite set of data points. Given $q\in\mathbb{N}$, our goal is to partition $\bm A$ into $q$ clusters $\bm A_1,\bm A_2,\ldots,\bm A_q$ satisfying
	\begin{enumerate}
		\item[(i)] $\bm A_i\neq\emptyset$, for all $i$,
		\item[(ii)] $\bm A_i\cap\bm A_j=\emptyset$, for all $i\neq j$,
		\item[(iii)] $\bm A=\bigcup_{i=1}^{q}\bm A_i$.
	\end{enumerate}
	
	Each cluster $\bm A_i$ is associated with a center point, denoted by $\bm c_i$. A data point $\bm a\in\bm A$ is assigned to cluster $\bm A_j$ whenever
	\[
	\|\bm a-\bm c_j\|
	=
	\min_{i=1,\ldots,q}\|\bm a-\bm c_i\|.
	\]
	In addition, we require the cluster centers $\bm c_i$, $i=1,2,\ldots,q$, to lie in a polyhedral set $\mathcal P\subset\mathbb{R}^n$. Such a clustering problem can be formulated as the following optimization problem \cite{bagirov2020}:
	
	\begin{align}\label{Clustering-P}
		\begin{split}
			&\quad\quad\,\,\min\,\, f_q(\bm C)\\&
			\text{s.t.} \quad \bm C=[\bm c_1,\bm c_2,\ldots,\bm c_q]\in\mathbb{R}^{n\times q},\\&
			\qquad \,\, \bm c_i\in\mathcal{P}, \quad i=1,2,\ldots,q.
		\end{split}
	\end{align}
	where
	\[
	f_q(\bm C):=\frac{1}{m}\sum_{j=1}^{m}\min_{i=1,\ldots,q}\|\bm a_j-\bm c_i\|.
	\]
	
	For $q>1$, the objective function is nonsmooth and nonconvex, and the problem has $n \times q$ decision variables.
	
	To generate a test instance, we randomly sampled $m=10,\!000$ two dimensional data points ($n=2$) uniformly from the unit disk centered at the origin. Moreover, the cluster centers $\bm c_i$, $i=1,\ldots,q$, were required to lie in the regular octagon $\mathcal{P}\subset\mathbb{R}^2$ with vertices
	\[
	(x_i,y_i):=0.7(\cos\theta_i,\sin\theta_i),\qquad
	\theta_i:=\frac{\pi}{8}+i\frac{\pi}{4},\quad
	i=0,1,\ldots,7.
	\]

	\begin{figure}
		\centering
		
		\includegraphics[width=\textwidth]{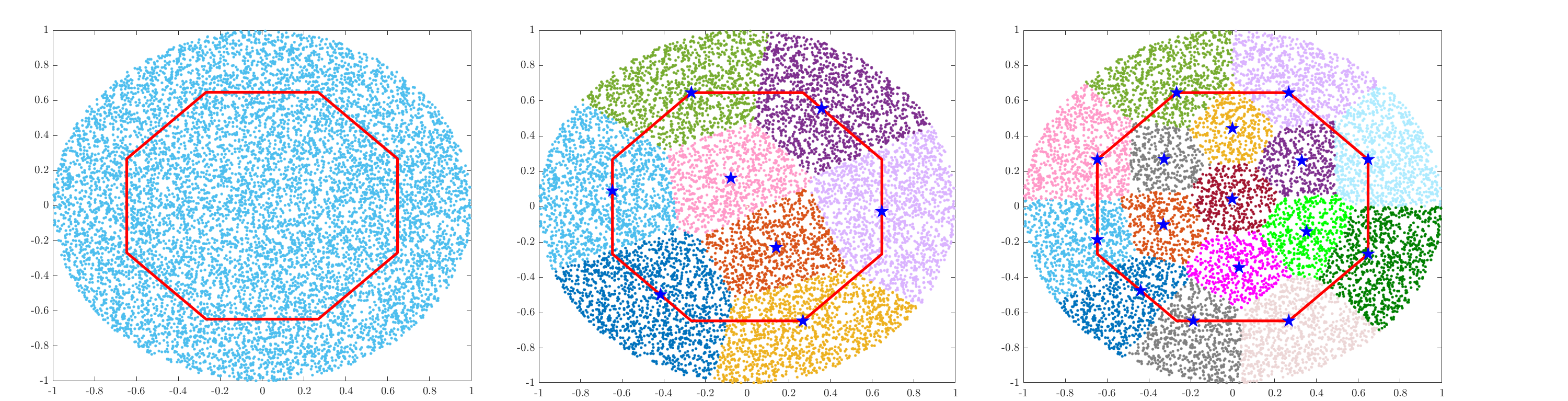}
		\caption{Left. the data set $\bm A$ and octagon $\mathcal{P}$. Middle. the obtained clusters $\bm A_i$ with the centers $\bm c_i$, for $q=8$. Right. the same for $q=16$. }\label{Fig4}
	\end{figure}
	
	The left plot of Figure~\ref{Fig4} illustrates the data set $\bm A$ together with the octagon $\mathcal{P}$. For $q=8$ and $16$, we then applied the proposed \textbf{PDSM} to  this instance of problem~\eqref{Clustering-P} using a randomly generated starting point and an optimality tolerance of $\tau=10^{-5}$. The resulting clusters $\bm A_i$, together with their corresponding centers, are shown in the middle and right plots of Figure~\ref{Fig4}. Moreover, Table~\ref{Table6} reports the computational performance of the \textbf{PDSM} on the clustering problem.

	\begin{table}[]
		\centering
		\caption{Computational performance of the \textbf{PDSM} to find center points $\bm c_i$, for $q=8$ and $16$.}\label{Table6}
		\resizebox{\textwidth}{!}{%
			\begin{tabular}{|lllllllllllllll|}
				\hline
				&\rule{0pt}{3ex}$q$ &  & Iter &  & Fun &  & Sub &  & $f_{best}$ &  & $v_f$ &  & Time(s) &  \\ \cline{2-2} \cline{4-4} \cline{6-6} \cline{8-8} \cline{10-10} \cline{12-12} \cline{14-14}
				&\rule{0pt}{3ex}8   &  & 82   &  & 741 &  & 65  &  & 0.0734   &  & 6E-6  &  & 4.84    &  \\
				& 16  &  & 96   &  & 358 &  & 79  &  & 0.0418   &  & 7E-6  &  & 3.86    &  \\ \hline
			\end{tabular}
		}
	\end{table}

	\section{Concluding Remarks}\label{Conclusion}
	We have developed a projected descent subgradient method for minimizing a weakly semismooth function over a closed and convex polyhedral set $C\subset\mathbb{R}^n$ and studied the global convergence behavior of the proposed method. Extending the method to a general closed and convex feasible set $C\subset\mathbb{R}^n$ is not straightforward. This is mainly due to the fact that, for a general closed and convex set, the projection operator $P_C:\mathbb{R}^n\to\mathbb{R}^n$ is not necessarily directionally differentiable, as shown by Kruskal~\cite{Kruskal1969}. Consequently, the weak semismoothness of $f$ cannot, in general, be inherited by the composite function $f\circ P_C$, which poses a challenge to establishing the finite convergence of Algorithm~\ref{projected-mifflin-line search}.
	
	Algorithm~\ref{main-alg} may be viewed as a basic framework that can be supplemented with several optional techniques. Instead of restarting the bundle of subgradients after each serious step, one may retain previously computed subgradients that still belong to the $\varepsilon$-subdifferential at the new iterate. The bundle can also be augmented using a subgradient sampling strategy, which may be effective in some situations. Although the effectiveness of this strategy has been problem-dependent in our observations, it can be a technique of choice for small-scale problems where subgradient evaluations are relatively inexpensive. If the number of consecutive null steps becomes large, storing the entire bundle may become impractical. In such cases, after solving subproblem~\eqref{least-norm}, the user may discard subgradients whose corresponding Lagrangian multipliers are sufficiently small.

\end{document}